\documentclass[a4paper,12pt]{article}
\usepackage{todonotes}
\usepackage{amsthm}
\usepackage[T1]{fontenc}
\usepackage{amssymb,amsmath,graphics}
\usepackage{bbm}
\usepackage{mathtools}
\usepackage{comment}

\newtheorem{theo}{Theorem}[section]

\newtheorem{prop}[theo]{Proposition}  
\newtheorem{coro}[theo]{Corollary}

\newtheorem{rema}[theo]{Remark}
\newtheorem{lema}[theo]{Lemma}

\newtheorem{defi}[theo]{Definition}

\newcommand{\zdarz}{X \textrm{ is 1/2-bad}}

\newcommand{\miu}{\mu_{I'}(i)}
\newcommand{\epp}{\va(I',X)}
\newcommand{\rod}{{\ccB_N}}
\newcommand{\all}{\mathcal{V}_N}

\def\be#1\ee{\begin{equation}#1\end{equation}}
\newcommand{\ba}{\begin{eqnarray} }
\newcommand{\ea}{\end{eqnarray} }
\newcommand{\eps}{\varepsilon}
\def\bt#1\et{\begin{theo}#1\end{theo}}
\def\bl#1\el{\begin{lema}#1\end{lema}}
\def\bp#1\ep{\begin{prop}#1\end{prop}}
\def\bd#1\ed{\begin{defi}#1\end{defi}}

 \global\long\def\cbr#1{\left\{  #1\right\}  }
 \global\long\def\rbr#1{\left(#1\right)}
 
\def\ccA{{\cal A}} 
\def\ccB{{\cal B}}

\def\ccE{{\cal E}}
\def\ccF{{\cal F}}
\def\ccG{{\cal G}}
\def\ccH{{\cal H}}
\def\ccI{{\cal I}}
\def\ccJ{{\cal J}}

\def\ccM{{\cal M}}
\def\ccN{{\cal N}}

\def\ccY{{\cal Y}}
\def\ccX{{\cal X}}

\def\va{\varepsilon}
\def\ra{\rightarrow}

\def\E{\mathbf{E}}

\def\P{\mathbf{P}}

\def\N{{\mathbb N}}

\def\R{{\mathbb R}}

\def\leq{\leqslant}
\def\geq{\geqslant}

\def\sm{\setminus}

\def\1{{\mathbbm 1}}
\providecommand{\keywords}[1]{\textbf{\textit{Keywords: }} #1}

\providecommand{\klas}[1]{\textbf{\textit{AMS MSC 2010: }} #1}
\begin{document}

\title{\bf The suprema of selector processes with the application to positive infinitely divisible processes\footnote{The authors were supported by National Science Centre, Poland grants 2019/35/B/ST1/04292 (W. Bednorz) and 2021/40/C/ST1/00330 (R. Meller). }}
\author{Witold Bednorz\footnote{Institute of Mathematics, University of Warsaw, Banacha 2, 02-097 Warsaw, Poland.},  Rafa\l{} Martynek\footnote{Institute of Mathematics, University of Warsaw, Banacha 2, 02-097 Warsaw, Poland.} and Rafa\l{} Meller\footnote{Institute of Mathematics, Polish Academy of Sciences, Śniadeckich 8,
00-656 Warsaw, Poland and Institute of Mathematics, University of Warsaw, Banacha 2, 02-097 Warsaw, Poland.}
}

\maketitle

\begin{abstract}
We provide an alternative proof of the recent result of Park and Pham \cite{Park} on the expected suprema of the positive selector and empirical processes, which we find useful in further generalizations we discuss. As an application, we obtain a result concerning positive infinitely divisible processes. \\
\keywords{expected suprema, positive processes, $p$-smallness.} \\
\klas{60G15, 60G17.}
\end{abstract}

\section{Introduction}
Recent results (see \cite{TG3}, \cite{Bed}) show that the size of the expected suprema of both empirical and infinitely divisible processes can be explained by decomposing the initial index set into a sum of two new index sets. The size of the process defined on one of them is characterized by Talagrand's gamma functionals, while the process defined on the other index set is a positive process. An understanding of the size of positive processes was lacking and some conjectures were made by Talagrand (see \cite{TG1, TG3}). Some of them have recently been discussed and proved in \cite{Park}.
\noindent The key new concept was the notion of a fragment, which was used in the proof of the Kahn-Kalai conjecture cf. \cite{Park}. We adopted this idea and managed to obtain a much simpler proof than the original one presented in \cite{Park}. Our proof is also more instructive for further generalizations. The primary research goal which motivates an alternative proof of Park and Pham result is to provide a small cover for processes $(Y_t)_{t\in T}$ for $T\subset\R_{+}^d$ and $Y_t=\sum_{i=1}^d t_iY_i$, where $Y_i$'s are non-negative, independent, but not necessarily identically distributed. The case of i.i.d. variables  is already a highly non-trivial result which is presented in \cite{up} (see also the end of Section $2$ below for the discussion of a result which goes beyond i.i.d. case and is also presented in \cite{up}). From this point of view, selector processes should be seen as a crucial, initial case one has to understand.  Interestingly, in the recent paper by Dubroff, Kahn and Park \cite{Kahn} the authors use our Key Lemma (Lemma \ref{zly} below) to establish a result concerning so-called expectation threshold and the fractional expectation threshold. Moreover, they call our result the quantitative strengthening of the result in \cite[Theorem 1.5]{Park} and use it for a very elegant computation in the direction of one of Talagrand's conjectures (\cite[Conjecture 6.3]{TG1}.

\noindent The overall goal is to relate the size of the stochastic process to the geometry of the index set. This leads to the very efficient method of chaining and, in turn, to characterization in terms of gamma functionals. The most famous result in this area is the Majorizing Measure Theorem (see e.g. \cite[Theorem 2.10.1]{TG3}) for Gaussian processes. One of the geometric consequences of this result is that the construction of the partition of the index set also provides, in a sense, "witnesses" to the size of the supremum. Let us briefly explain this concept (cf. the discussion after Theorem 2.12.2 and Section 13.2 in \cite{TG3}). Let $(G_t)_{t\in T}$ be a centred Gaussian process on some index set $T$.  By the Majorizing Measure Theorem one can find a jointly Gaussian sequence $(u_k)_{k\geq1}$ such that for some universal constant $L>0$ it holds that
$$\left\{\sup_{t\in T}|G_t|\geq L\E\sup_{t\in T}|G_t|\right\}\subset \bigcup_{k\geq 1}\left\{u_k\geq 1\right\}$$
and $$\sum_{k\geq1}\P(u_k\geq 1)\leq\frac{1}{2}.$$
Let us emphasize that the above statement is weaker than the Majorizing Measure Theorem in the sense that it follows relatively easily from it, but the other direction is not at all clear. The basis of the chaining method is a result of Sudakov's type bounds, which guarantees that if the points in the index set are well separated, then one can control the expected supremum of the process from below. Since we do not have such a result in the context of the selector process, chaining seems infeasible in this case. Therefore, we want to characterize $\sup_{t\in T} X_t$ of certain non-negative stochastic processes in the sense of the above representation.\\
\noindent As mentioned before, the main motivation for simplifying the proof of Theorem \ref{main1} is to find new methods of bounding $\E\sup_{t\in T} X_t$, where $(X_t)_{t\in T}$ is a nonnegative process. Due to the use of multisets in \cite{Park} in the case of empirical processes, it is difficult to notice how this method can serve in different settings. The main difference between our approach and the original proof in \cite{Park} is the method of counting of so-called bad sets. In particular, we introduce a kind of a threshold function (see Definition \ref{jotyiepsilony} and the beginning of the proof of Lemma \ref{lematepsilon}), which turns out to be an appropriate tool for generalizations. As an example, we give a result for permutationally invariant processes in the Appendix. Although the further generalizations in \cite{up} are much more complicated than the proof for selector processes, the appropriate version of the threshold function plays a crucial role there. \\
\noindent

\section{Results and motivation}\label{resul}
The letters $K,C$ stand for universal constants. We adopt the convention that $K$ may differ at each occurrence.
As we mentioned, the first goal is to present a simplified proof of the result in \cite{Park} concerning a positive selector process. 
Consider a sequence $(\delta_i)_{i=1}^{n}$ of independent Bernoulli random variables with 
$$\P(\delta_i=1)=p=1-\P(\delta_i=0)$$ 
for some $p\in(0,1)$ and a set $T$ of sequences $t=(t_i)_{i=1}^{n}$ with $t_i\geq0$. We assume that $T$ is finite. Let 
$$\delta(T):=\E\sup_{t\in T}\sum_{i=1}^{n}t_i\delta_i.$$
We use the notation $[n]:=\{1, 2, \dots, n\}$. To describe the size of $\delta(T)$, we want to construct "explicit witnesses" of the event
$$\left\{\sup_{t\in T}\sum_{i=1}^{n}t_i\delta_i\geq K\delta(T)\right\}$$
whose measure is small. To formalize this goal let us recall the notions of up-sets and $p$-smallness. 
Let $I\subseteq[n]$. Then the \textit{up-set generated by $I$} is defined as $$\langle I\rangle:=\left\{J\subseteq[n]: J\supseteq I\right\},$$ 
and we say that the \textit{collection $\ccF$ of subsets of $[n]$ is $p$-small} if there exists a collection $\ccG$ of subsets of $[n]$ such that 
$$\ccF\subseteq\langle\ccG\rangle:=\bigcup_{I\in\ccG}\langle I\rangle,$$
and
$$\sum_{I\in\ccG}p^{|I|}\leq\frac{1}{2}.$$
We refer to $\ccG$ as the cover of $\ccF$.
We prove the following result describing the size of $\delta(T)$.
\begin{theo}\label{main1}
The family
\be\label{fam}
\left\{I\subset [n]: \sup_{t\in T}\sum_{i\in I} t_i\geq K\delta(T)\right\}
\ee
is $p$-small. One may take $K=221$.
\end{theo}
\begin{rema}\label{malep}
Clearly $\sup_{t\in T}\sum_{i=1}^n t_i\delta_i\leq \sup_{t\in T}\sum_{i=1}^n t_i\leq \frac{1}{p} \delta(T)$. Thus the family in \eqref{fam} is empty for $K>\frac{1}{p}$. This shows that Theorem \ref{main1} is interesting only when $p$ is small.
\end{rema}
\noindent We would like to point out that the proof of Theorem \ref{main1} is existential, we do not construct this family explicitly. This is a  common situation in this field. Usually, the constructive results are derived only in some special cases. The best example is the Talagrand's Majorizing Measure Theorem for which explicit construction of the majorizing measure is known only for simplexes and ellipses (and a few exotic examples). 
\noindent The main new consequence that we deduce from Theorem \ref{main1} is the analogous result for positive infinitely divisible processes (see \cite[Definition 12.3.4]{TG3}). Notice that following \cite[Chapter 12]{TG3} we take Rosi\'nski's  series representation of infinitely divisible processes as their definition which uses a Poisson point process (see e.g. \cite[Section 12.1]{TG3}).\\ 
\noindent Consider a finite set $T$ and a measurable space $\R^T$, provided with the cylindrical $\sigma$-algebra
\[ \sigma \left(\{x\in \R^T: x(t)\in A\}:t\in T, A\in \mathcal{B}(\R) \right), \]
i.e. $\sigma$-algebra generated by the coordinate functions. Let $\nu$ be a $\sigma$-finite measure on $\R^T$ with $\nu(0)=0$. Assume that for any $t\in T$  $$\int_{\R^T} |x(t)|\wedge 1 d\nu(x)<\infty,$$
where $\wedge$ is the minimum of two values. The main additional assumption we make about the measure $\nu$ is that it has no atoms. Consider a Poisson point process $(Y_i)_{i\geq 1}$ on $\R^{T}$ 
with intensity measure $\nu$. Then we define \textit{a positive infinitely divisible process with L\'evy measure $\nu$} as
\[
|X|_t:=\sum_{i\geq 1}|Y_i(t)|.
\] 
The standard simplification is to consider $T$ as a space of functions on $\R^T$ by associating the element $t\in T$ with the function $x\mapsto x(t)$ on $\R^T$, so that we can write $t(Y_i)$ instead of $Y_i(t)$.
Thus, we have
\[\E\sup_{t\in T}|X|_t=\E\sup_{t\in T}\sum_{i\geq 1}|t(Y_i)|.
\]

\bt \label{twpod}
Let $|X|_t$ be a positive, infinitely divisible process with a non-atomic L\'evy measure $\nu$, indexed by a finite set $T$. Then there exists a family $\ccF$ of pairs $(g,u)$, where each $g:\R^{T}\ra \R_{+}$
is measurable (with respect to the cylindrical $\sigma$-algebra) and $u\geq 0$ such that
\begin{equation}\label{event}
\cbr{\sup_{t\in T}|X|_t \geq K\E\sup_{t\in T}|X|_t  }\subset \bigcup_{(g,u)\in \ccF}\{|X|_g\geq u\},
\end{equation} 
and
\begin{equation}\label{delta}
\sum_{(g,u)\in \ccF}\P(|X|_g\geq u)\leq \frac{1}{2}.
\end{equation}
\et
\noindent Obviously, by the Markov inequality, the probability of the event described in \eqref{event} is less than $1/K$. But the goal is to find "explicit witnesses" that this event is small.

\noindent The third setting that we work with is the one of positive empirical processes. 
Although we must assume that the underlying measure $\mu$ is non-atomic, we find it instructive to see that the result can be obtained without the use of multisets, unlike \cite{Park}. We believe this makes the application of the result for the selectors more transparent. For the result in full generality we refer to \cite{Park}.

\begin{theo}\label{twemp}
Let $(E,\rho)$ be some metric space, $d>0$ be a fixed integer. Suppose $\mu$ is a probability measure on $E$ with no atoms.  Let $X_1, \dots, X_d$ be independent random variables distributed according to $\mu$. Consider $T$, a finite class of nonnegative Borel measurable functions on $E$.  Then there exists a family $\ccF$ of pairs $(g,u)$ where each $g:E\ra \R_{+}$
is Borel measurable and $u\geq 0$ such that
\begin{equation}\label{eventem}
\cbr{\sup_{t\in T}\sum_{i=1}^d t(X_i)\geq K\E\sup_{t\in T}\sum_{i=1}^d t(X_i)}\subset \bigcup_{(g,u)\in \ccF}\cbr{\frac{1}{d}\sum_{i=1}^d g(X_i)\geq u}
\end{equation} 
and
\begin{equation}\label{deltaem}
\sum_{(g,u)\in \ccF}\P\left(\frac{1}{d}\sum_{i=1}^d g(X_i)\geq u\right)\leq \frac{1}{2}.
\end{equation}
\end{theo}

\noindent The metric structure is irrelevant for the formulation of the theorem.
However, it is crucial for our proof. Roughly speaking, we need to decompose set $A \subset E$ into many subsets of almost equal measure (see Step~3 of
the proof). It is not clear whether such a construction is possible in non-metric spaces.

\noindent As we show in \cite{up}, the assumption that the measures have no atoms is not necessary in both Theorems \ref{twpod} and \ref{twemp}. We also suspect that the assumption that $T$ is finite in Theorems \ref{twpod} or \ref{twemp} is redundant, but it is unclear how to extend the results to infinite index sets.\\
\noindent We finish this section with a short discussion of a possible direction for the future work. In \cite{up} we manage, by using techniques developed in this paper, to provide a small cover for
\begin{equation}\label{conj1}
\cbr{\sup_{f\in T}\sum_{i=1}^d f_i(X_i) \geq K\E\sup_{f\in T}\sum_{i=1}^d f_i(X_i)},
\end{equation}
where $X_1,\ldots,X_d$ are non-negative i.i.d random variables and $T$ is a finite set of $d$-tuples of non-negative functions. We believe that this result can be of use in some other important problems. Let us briefly discuss one of them. The problem of finding tight bounds on the operator norm of random matrices is of great importance in probability theory and functional analysis. The most important case of non-homogeneous Gaussian matrices has been solved by R. van Handel, R. Lata{\l}a, and P. Youssef cf. \cite{Lat}. But the equally important question about non-homogeneous Rademacher matrices is still open.  The task is to find a two-sided bound for the quantities 
 \begin{equation}\label{loca12}
   \E \sup_{x,y\in B^n_2} \sum_{ij} a_{ij} x_i y_j \eps_{ij} = \E \sup_{x\in B^n_2} \sqrt{\sum_i (\sum_j  a_{ij}x_j\eps_{ij})^2}\approx\sqrt{\E \sup_{x\in B^n_2} \sum_i (\sum_j  a_{ij}x_j\eps_{ij})^2}  
 \end{equation}
where $a_{ij}$ are real numbers and $\eps_{ij}$ are independent, symmetric $\pm1$ random variables. The quantity $\sup_{x\in B^n_2} \sum_i (\sum_j  a_{ij}x_j\eps_{ij})^2$ can easily be expressed as $\sup_{f\in T}\sum_{i=1}^d f_i(X_i)$. The existence of a small cover of the set (\ref{conj1}) may result in a better understanding of the problem of estimating the operator norm of Rademacher matrices. In particular, it may lead to a counterpart of Theorem \ref{main2} in this case, potentially providing new results.For a broader discussion of estimates of \eqref{loca12}, we refer the reader to \cite{latop,opnormber}; see also \cite{Meller} for the case of more general random variables.\\
\noindent  The organization of the paper is the following. In the next three sections, we prove, respectively, Theorem
\ref{main1}, result on infinitely divisible processes and on empirical processes. Two of them are present in
\cite{Park} while the one concerning positive infinitely divisible processes is our main application of the result
about selector processes. In the Appendix we present the result for processes invariant under permutations of the index set, which aims to show how the main argument can be adjusted for further applications.

\section{Positive selector process}
Theorem \ref{main1} can be reformulated as follows (see the proof of Theorem \ref{malarodzina} below or \cite[Section 2.1]{Park}).

\begin{theo}\label{main2}
Let $\ccF$ be a family of subsets of $[n]$ which is not $p$-small. Suppose that for each $I\in\ccF$ we are given coefficients 
$(\mu_{I}(i))_{i\in I}$ with $\mu_I(i)\geq0$ and 
$$\mu_I(I):=\sum_{i\in I}\mu_I(i)= 1.$$ Then 
$$\E\sup_{I\in\ccF}\sum_{i\in I}\mu_{I}(i)\delta_i\geq\frac{1}{220}.$$
\end{theo}
\noindent Theorem \ref{main2} is obvious if constant $\frac{1}{220}$ is replaced by $p$, where $p=\P(\delta_i=1)$ (even without the essential assumption that $\ccF$ is not $p$-small). Thus it is interesting only for small values of $p$. (cf. Remark \ref{malep}).
\begin{rema}\label{pn}
Obviously, empty set does not contribute to the supremum over $I\in\ccF$, so we may assume that $\emptyset \notin \ccF$. For any such family $\ccF$ of subsets of $[n]$ we have $\ccF \subseteq \bigcup_{i\leq n} \langle \{i\} \rangle$. Thus, Theorem \ref{main2} is trivial when $np<1/2$ (in this case any family $\ccF$ is $p$-small). Thus, with no loss in generality, we may assume that $np\geq 1/2$.
\end{rema}
\noindent Our goal now is to prove Theorem \ref{main2}. To this end, we introduce the pivotal definition.
\begin{defi}
We say that a set $X\subseteq [n]$ is $c$-bad ($c>0$) if 
\begin{equation}\label{bad}
\sup_{I\in\ccF}\mu_I(X\cap I)<c.
\end{equation}
\end{defi}
\noindent By saying that $X\subseteq [n]$ is random we mean that
\begin{equation}\label{randomset}
X=\{i\in [n]: \delta_i'=1\},    
\end{equation}
where $(\delta_i')_{i\leq n}$ are independent and identically distributed Bernoulli variables with arbitrary probability of success. 
Notice that the distribution of $X$ conditioned on its cardinality equal to $m$ is uniform over $m$-element subsets of $[n]$. 
\noindent We deduce Theorem \ref{main2} from the following.
\begin{lema}[Key Lemma]\label{zly}
Let $\ccF$ be a family of subsets of $[n]$ which is not $p$-small. 
For any $m \leq n$ we have
\[\P(\zdarz \;\big|\; |X|=m) \leq \sum_{t=1}^n \left(4\frac{np}{m} \right)^t,\]
where a bad set is defined in \eqref{bad}.
\end{lema}

\begin{proof}[Proof of Theorem \ref{main2}]
By Remark \ref{pn} we can assume that $np\geq1/2$. Thus we may further assume that $Cnp$ is an integer and  $9\leq C\leq 11$.
Observe that for each $i$ we can represent $\delta_i=\delta_i'\delta_i''$, where $(\delta_i')_{i=1}^{n},(\delta_i'')_{i=1}^{n}$ are independent Bernoulli random variables with $\P(\delta_i'=1)=Cp, \P(\delta_i''=1)=\frac{1}{C}$. 
By Jensen's inequality
\begin{align}
 \nonumber\E \sup_{I \in \ccF} \sum_{i \in I} \mu_I(i)\delta_i&=\E \sup_{I \in \ccF} \sum_{i \in I} \mu_I(i)\delta_i'\delta_i'' \\
&\geq \E \sup_{I \in \ccF} \sum_{i \in I} \mu_I(i)\delta_i'\E \delta_i''\geq\frac{1}{11}\E \sup_{I \in \ccF} \sum_{i \in I} \mu_I(i)\delta_i'.\label{ala123}
\end{align}
Lemma \ref{zly} applied for $m\geq Cnp$ gives
\begin{multline}
\E\left( \sup_{I\in \ccF} \sum_{i \in I} \mu_I(i) \delta_i' \mid \sum_{i=1}^n \delta_i'=m\right)\geq \frac{1}{2} \P\left(X\; \textit{is not 1/2-bad} \mid \sum_{i=1}^n \delta_i'=m \right) \label{fin1} \\
\geq \frac{1}{2}\left(1-\sum_{t=1}^n \left(4\frac{np}{m} \right)^t \right)    \geq  \frac{1}{2}\left(1- \frac{np}{m} \frac{4}{1-\frac{4np}{m}}\right)    \geq \frac{1}{2}\left(1-\frac{4}{C-4} \right)\geq \frac{1}{10}.
\end{multline}
So, by conditioning
\begin{align}
 \nonumber \E \sup_{I\in \ccF} \sum_{i \in I} \mu_I(i) \delta_i '&\geq \sum_{m\geq Cpn} \E\left( \sup_{I\in \ccF} \sum_{i \in I} \mu_I(i) \delta_i' \mid \sum_{i=1}^n \delta_i'=m\right) \P\left(\sum_{i=1}^n \delta_i'=m\right)\\
&\geq \frac{
1}{10} \P\left(\sum_{i=1}^n \delta_i'\geq Cpn\right)\geq \frac{1}{20},\label{ala321} 
\end{align}
where in the last inequality we used \cite{Lord} (we recall that $Cnp$ is an integer). The assertion follows by \eqref{ala123} and \eqref{ala321}.
\end{proof}

\begin{rema}\label{rem-mniejszy-zbior}
Let us observe that Theorem \ref{main1} implies that the event
\[\cbr{\sup_{t\in T} \sum_{i=1}^n t_i \delta_i \geq 221C \delta(T) } \]
is $Cp$-small. To see this let $\delta_i',\delta_i''$ be as in the proof of Theorem \ref{main2}. Denote $\delta'(T):=\E \sup_{t\in T} \sum_{i=1}^n t_i \delta_i'$. By Theorem \ref{main1} the event
\[\cbr{\sup_{t\in T} \sum_{i=1}^n t_i \delta'_i \geq 221\delta'(T) }\]
is $Cp$-small and trivially $$\cbr{\sup_{t\in T} \sum_{i=1}^n t_i \delta'_i \geq 221\delta'(T) }\supset \cbr{\sup_{t\in T} \sum_{i=1}^n t_i \delta_i \geq 221\delta'(T)}.$$
The inequality \eqref{ala123} implies that $\delta'(T)\leq C \delta(T)$. Thus, we obtain
\[\cbr{\sup_{t\in T} \sum_{i=1}^n t_i \delta_i \geq 221C \delta(T)} \subset  \cbr{\sup_{t\in T} \sum_{i=1}^n t_i \delta'_i \geq 221\delta'(T)},\]
and the result follows.
\end{rema}
\subsection{Proof of the Key Lemma}
The overall strategy for the proof is to construct the cover
of $\ccF$ based on the class of bad sets with fixed cardinality and show that its measure is appropriately
small. To achieve this goal we need sets closely related to so-called fragments (cf. \cite[Definition 2.3]{Park} or \cite{Park1}, \cite{Park2} for more context on the idea of the fragment). We introduce it after the next result.
Let $I_j$ be the subset of $I$ consisting of $j$ elements with the largest values of coefficients i.e. 
$$I_j=\left\{i_l\in I, \;1\leq l\leq j:\;\; \mu_I(i_1)\geq \mu_I(i_2)\geq\dots\geq\mu_I(i_j)\geq \sup_{i\in I\setminus \{i_1,\ldots,i_j \}}\mu_I(i) \right\},$$
where we set the convention that $I_0=\emptyset$ and for $j\geq |I|$ we have $I_j=I$. Also, $I_j^c=I\sm I_j$. 
\begin{lema}\label{lema1}
 Let $X\subseteq [n]$ be $c$-bad (recall \eqref{bad}) for some $c>0$
Then, for any $I\in\ccF$ there exists a number $j(I,X)$ such that 
\begin{equation}\label{cardinality}
|I_{j(I, X)}\cap X|<c|I_{j(I, X)}|.
\end{equation}
\end{lema}
\begin{proof}
Assume that $X$ is non-empty, otherwise, there is nothing to prove. Fix $I\in\ccF$. For $\va\in[0,1]$ consider a mapping
$$\va\mapsto\sum_{i\in I\cap X}\mu_I(i)\wedge\va-c\sum_{i\in I}\mu_I(i)\wedge\va$$
and observe that it is continuous. Moreover, for $\va=1$ we have by (\ref{bad}) that $$\sum_{i\in I\cap X}\mu_I(i)-c\sum_{i\in I}\mu_I(i)<c-c\mu_I(I)=c-c=0,$$
and obviously $\va=0$ is mapped to $0$. This means that there exists the largest value $\va(I,X)\in[0,1)$ such that 
\begin{equation}\label{last}
\sum_{i\in I\cap X}\mu_I(i)\wedge\va(I,X)\geq c\sum_{i\in I}\mu_I(i)\wedge\va(I,X)
\end{equation}
and for $\va>\va(I,X)$ the strict reverse inequality holds. Now, observe that for sufficiently small $\delta>0$ we have
\begin{align}\label{loc78}
    \sum_{i\in I\cap X}\mu_I(i)\wedge(\va(I,X)+\delta)&=\sum_{i\in I\cap X}\mu_I(i)\wedge\va(I,X)+\delta\sum_{i\in I\cap X}\1_{\{\mu_I(i)>\va(I,X)\}}
    \end{align}
    and
 \begin{align}\label{79}   \sum_{i\in I}\mu_I(i)\wedge(\va(I,X)+\delta)&=\sum_{i\in I}\mu_I(i)\wedge\va(I,X)+\delta\sum_{i\in I}\1_{\{\mu_I(i)>\va(I,X)\}}.
\end{align}
By the definition of $\va(I,X)$
\[\sum_{i\in I\cap X}\mu_I(i)\wedge(\va(I,X)+\delta) < c \sum_{i\in I}\mu_I(i)\wedge(\va(I,X)+\delta).  \]
We combine the above  inequality with \eqref{loc78} and \eqref{79} to get that
$$\sum_{i\in I\cap X}\mu_I(i)\wedge\va(I,X)+\delta\sum_{i\in I\cap X}\1_{\{\mu_I(i)>\va(I,X)\}}<c\sum_{i\in I}\mu_I(i)\wedge\va(I,X)+c\delta\sum_{i\in I}\1_{\{\mu_I(i)>\va(I,X)\}},$$
but because of (\ref{last}) it follows that
$$\sum_{i\in I\cap X}\1_{\{\mu_I(i)>\va(I,X)\}}<c\sum_{i\in I}\1_{\{\mu_I(i)>\va(I,X)\}}.$$ Let 
\begin{equation}\label{joty}
\tilde{I}:=\{i\in I:\; \mu_I(i)>\va(I,X)\}\;\; \mbox{and}\;\; j(I,X):=|\tilde{I}|.
\end{equation}
The result follows, since obviously
\begin{equation}\label{wspol}
I_{j(I,X)}=\tilde{I}.
\end{equation} 
\end{proof}
\noindent Since the quantities $\varepsilon(I,X)$ and $j(I,X)$ play a central role in our
construction, we introduce them in a separate definition.
\begin{defi}\label{jotyiepsilony}
Let $X\subseteq [n]$ be $c$-bad (recall \eqref{bad}) for some $c>0$ and $I\in\ccF$. Let
$$F(\varepsilon)=\sum_{i\in I\cap X}\mu_I(i)\wedge\va-c\sum_{i\in I}\mu_I(i)\wedge\va.$$
We define $\varepsilon(I, X)$ as the largest argument $\varepsilon$ for which $F(\varepsilon)\geq0$. 
We also define 
\begin{align}\label{eq:defj}
 j(I,X):=|\{i\in I:\; \mu_I(i)>\va(I,X)\}|.   
\end{align}
\end{defi}
\begin{defi}[Witness]\label{Fragment}
Fix a $c$-bad set $X\subset [n]$. To each $I\in \ccF$ we associate the number $j(I,X)$ as in Definition \ref{jotyiepsilony}. Now, fix $I\in \ccF$. Consider sets $\hat{I}\in \ccF$  such that
\begin{equation}\label{zawieranie}
\hat{I}_{j(\hat{I},X)} \setminus X \subset I \setminus X.
\end{equation}
Among all $\hat{I}$ for which \eqref{zawieranie} holds we consider the ones for which $j(\hat{I},X)$ is the smallest possible. Among the latter, we pick any $\hat{I}$ such that $\hat{I}_{j(\hat{I},X)}\setminus X$ has the minimal number of elements and denote it by $I'=I'(I,X)$. We refer to $I'$ as the $(I,X)$ minimal set, or simply as the minimal set. We also define a $(I,X)$-witness as
$$W(I,X):=I'_{j(I',X)}.$$
\end{defi}
\noindent If $j=j(I',X)$ and $t=|I'_j\setminus X|$ then it is easy to see that by Lemma \ref{lema1} we have
\begin{equation}\label{ogrt}
j\geq t\geq (1-c)j.
\end{equation}
The aforementioned cover of $\ccF$ we are interested in is given for a fixed $c$-bad $X$ by
\begin{equation}\label{cover}
\ccG(X)=\{W(I, X)\setminus X: I\in\ccF\}.
\end{equation}
\begin{rema}
Indeed, this is a cover of $\ccF$ since $W(I,X)\setminus X\subseteq I$ for every $I\in\ccF$. Let us note that sets $W(I,X)\setminus X$ are analogs of fragments in \cite{Park}. However, the notion of the witness is more suitable for our proof.
\end{rema}
Since for any $X$, which is $1/2-$bad, $\ccG(X)$ is a cover of $\ccF$ and $\ccF$ is not $p-$small, we have 
$$\frac{1}{2}\leq \sum_{G\in \ccG(X)} p^{|G|}.$$
In addition, the number of sets $X\subset[n]$ with $|X|=m$ that are $1/2-$bad is $$\P(\zdarz \mid |X|=m) \binom{n}{m}.$$ Therefore, we have
\begin{equation}\label{nier1}
\frac{1}{2}\P(\zdarz \mid |X|=m) \binom{n}{m} \leq \sum_{\substack{\zdarz \\ |X|=m}} \sum_{G\in \ccG(X)} p^{|G|}.
\end{equation}

\noindent In order to bound the above double sum we need to control from above the number of repetitions of the witnesses. It is deduced in Corollary \ref{zlicz} from the following two results.

\begin{lema}\label{pom}
Take any $c$-bad sets $X,Y$ and $I,J \in \ccF$. Denote $I'=I'(I,X)$ and $J'=J'(J,Y)$ (the minimal sets). Assume that 
\begin{enumerate}
\item $j:=j(I',X)=j(J',Y)$ (recall \eqref{eq:defj}),\label{jed}
\item $Z:=I'_j\cup X=J'_j \cup Y$, \label{dwa}
\item $t:=|I'_j\setminus X|=|J'_j \setminus Y|$. \label{trzy}
\end{enumerate}
Then $$\sum_{i\in I'\cap Y} \mu_{I'}(i)\wedge \va(I',X)\geq c \sum_{i\in I'}\miu \wedge \epp,$$ where $\epp$ is defined as in \eqref{last}.
\end{lema}
\begin{rema}\label{pom1}
Observe that, by the definition, $\va(I', Y)$ is the greatest number for which the above inequality is satisfied. This means that $\va(I',Y)\geq \va(I',X)$. This is the key consequence of Lemma \ref{pom}. 
\end{rema}
\begin{proof}[Proof of Lemma \ref{pom}.]
Let  $(I'_j)^c:=I'\setminus I'_j$ (the complement of $I'_j$ in $I'$) where we recall (see \eqref{wspol})
\[I'_j=\{i\in I': \mu_{I'}(i)>\eps(I',X)\}.\]
By the above and since $Y\subset Z$ we have
\begin{align*}
\sum_{i\in I'\cap Y} \mu_{I'}(i)\wedge &\va(I',X)=\sum_{i\in (I'_j)^c\cap Y} \mu_{I'}(i)+\sum_{i\in I'_j\cap Y}\epp\\
&=\sum_{i\in (I'_j)^c\cap Z} \mu_{I'}(i)-\sum_{i\in (I'_j)^c\cap (Z\setminus Y)} \mu_{I'}(i)+\epp|I'_j\cap Y|\\
&=\sum_{i\in (I'_j)^c\cap Z} \mu_{I'}(i)-\sum_{i\in (I'_j)^c\cap (J'_j\setminus Y)} \mu_{I'}(i)+\epp\left(|I'_j\cap Z|-|I'_j\cap(Z\setminus Y)| \right)\\
&=\sum_{i\in (I'_j)^c\cap Z} \mu_{I'}(i)-\sum_{i\in (I'_j)^c\cap ((J'_j\setminus I'_j)\setminus Y)} \mu_{I'}(i)+\epp\left(j-|I'_j\cap(J'_j\setminus Y)| \right),
\end{align*}
where in the last two equalities we used the assumption \ref{dwa}. By the assumption \ref{trzy}, the last line is equal to

\begin{align*}
&\sum_{i\in (I'_j)^c\cap Z} \mu_{I'}(i)-\sum_{i\in (I'_j)^c\cap ((J'_j\setminus I'_j)\setminus Y)} \mu_{I'}(i)+\epp\left(j-t+|(J'_j\setminus I'_j)\setminus Y|\right)\\
&\geq \sum_{i\in (I'_j)^c\cap Z} \mu_{I'}(i) +\epp(j-t).
\end{align*}
The last inequality holds, because  if $i\in (I'_j)^c$ then by the definition of $I'_j$ (see \eqref{wspol}) we have that $\mu_{I'}(i)<\epp$. Now, using assumption \ref{trzy} once more,
\begin{align*}
\sum_{i\in (I'_j)^c\cap Z} \mu_{I'}(i) +\epp(j-t)&=\sum_{i\in (I'_j)^c\cap Z} \mu_{I'}(i) +\epp|I'_j \cap X|\\
&=\sum_{i\in I'\cap X}\miu \wedge \epp\geq c\sum_{i\in I'} \miu \wedge \epp.
\end{align*}
The last inequality follows again from the definition \eqref{last} of $\epp$.
\end{proof}

\begin{lema}\label{zaw}
Under the assumptions of Lemma \ref{pom} it is true that $I'_j\setminus Y=J_j'\setminus Y$ (i.e. $ W(I,X) \setminus Y=W(J,Y) \setminus Y $).
\end{lema}
\begin{proof}
By the construction of the witness, together with assumption \ref{dwa} of Lemma \ref{pom}, we have
\begin{equation}\label{loc1}
J\setminus Y \supset J'_j\setminus Y=(J'_j\cup Y) \setminus Y=(I'_j\cup X) \setminus Y \supset I'_j \setminus Y.\
\end{equation}
Thus we must have
\begin{equation}\label{lowb}
j_0:=j(I',Y)\geq j,
\end{equation} 
otherwise we would have
\[I'_{j_0}\setminus Y \subset  I'_j \setminus Y \subset J\setminus Y ,\]
which would contradict the optimality of $j=j(J',Y)$. Recall (\ref{wspol}) and suppose that $j_0=j(I',Y)>j$. The larger cardinality obviously indicates the lower threshold, therefore we have $\va(I',Y)<\va(I',X)$. This, however, contradicts Lemma~\ref{pom} (see Remark~\ref{pom1}); hence, we must have $j_0 = j$.  This implies that $|I_j'\setminus Y|=|J'_j\setminus Y|$ by the construction of the witness. Since these two sets have the same cardinality and one is a subset of the other by~\eqref{loc1}, it follows that in fact $I'_j \setminus Y = J'_j \setminus Y$.
\end{proof}
\begin{coro}\label{zlicz}
Fix $m,t$ and  $Z\subset [n]$ such that $|Z|=m+t$. Fix also $X\subset Z$ which is $c$-bad and $|X|=m$.
Then for any $t\leq j\leq n$ 
\begin{equation}\label{zbiorit}
    \left|\{W(I,X)\setminus X : I\in \ccF,\ |W(I,X)|=j,\ Z=W(I,X)\cup X,\ |W(I,X)\setminus X|=t   \}\right| \leq \binom{j}{t}. 
\end{equation}

\end{coro}
\begin{proof}
Fix any $X,I$ satisfying conditions in \eqref{zbiorit}. Take any other pair $Y,J$ which also satisfies the same conditions. By Lemma \ref{zaw}, $W(J,Y)\setminus Y \subset W(I,X) $, so there are at most ${j \choose t}$ choices for $W(J,Y)\setminus Y $.
\end{proof}
\noindent We are in the position to bound the expression in \eqref{nier1} and prove the Key Lemma.
\begin{proof}[Proof of Lemma \ref{zly}]
We recall the notation  $I'_j=W(I,X)$, where $j=j(I,X)$ (cf. \eqref{joty}).
Recall the construction of the cover \eqref{cover}. Then  \eqref{ogrt} together with Corollary \ref{zlicz} imply that
\begin{align}
\sum_{\substack{\zdarz \\ |X|=m}} \sum_{G\in \ccG(X)} p^{|G|}&=\sum_{j\geq 1} \sum_{j/2\leq t \leq j} \sum_{|Z|=m+t} p^t|\{W(I,X)\setminus X: Z=I'_j\cup X,\ |I'_j \setminus X|=t\}| \nonumber\\
&\leq \sum_{j\geq 1} \sum_{j/2\leq t \leq j} \sum_{|Z|=m+t} p^t\binom{j}{t}=\sum_{t=1}^n \sum_{j=t}^{2t}  \binom{j}{t}\binom{n}{m+t}p^t.\label{nier2}
\end{align}
Clearly
\[\sum_{j=t}^{2t} \binom{j}{t}=\binom{2t+1}{ t+1}\leq \frac{1}{2}2^{2t+1}=4^t\]
and
\[\frac{\binom{n}{m+t}}{\binom{n}{m}}\leq \left( \frac{n}{m} \right)^t.\]
Those two inequalities applied to \eqref{nier2} give
\begin{align*}
\sum_{\substack{\zdarz \\ |X|=m}} \sum_{G\in \ccG(X)} p^{|G|}&\leq \binom{n}{m} \sum_{t=1}^n \left(4 \frac{np}{m}\right)^t.
\end{align*}
The above together with \eqref{nier1} imply the assertion.
\end{proof}

\section{Positive infinitely divisible processes}
 The following definition is an analog of $p$-smallness for an event in $\R^T$.
\begin{defi}\label{deltasmall}
We say that the random event $A$ is $\delta$-small if there exists a family $\ccF$ of pairs $(g,u)$, where $g:\R^{T}\ra \R_{+}$
is measurable (with respect to the cylindrical  $\sigma$-algebra) and $u\geq 0$ such that
$$
A\subset \bigcup_{(g,u)\in \ccF}\{|X|_g\geq u\}$$
and $$\sum_{(g,u)\in \ccF}\P(|X|_g\geq u)\leq \delta. 
$$
\end{defi}
\noindent In other words Theorem \ref{twpod} states that the event $\cbr{\sup_{t\in T}|X|_t \geq K\E\sup_{t\in T}|X|_t  }$ is $1/2$-small.
\begin{proof}[Proof of Theorem \ref{twpod}]
Note that if $S=\E\ \sup_{t\in T} |X|_t = \infty$ then the left hand side of (\ref{event}) is empty, so any small cover will work in this case. In particular, we can assume that
\begin{equation}\label{nuogr}
 \sup_{t\in T} \int_{\R^T} |t(x)|d\nu(x) =\sup_{t\in T} \E\sum_{i\geq1}|t(Y_i)|= \sup_{t\in T} \E|X|_t\leq \E \sup_{t\in T}|X|_t=S  < \infty.
\end{equation}
The proof is based on successive simplifications, which ultimately allow us to apply Theorem~\ref{main1}.\\
\noindent {\bf Step 1.} Consider some fixed large integer $N$ and functions $f,h:\R \to \R$ given by 
\[f(s)=|s|\1_{\{|s|\geq 2^{-N}\}},\ h(s)=|s|\1_{\{|s|<2^{-N}\}}.\]
We split the process into two parts
\begin{align*}
  |X|_t=\sum_{i\geq 1} |t(Y_i)|\1_{\{|t(Y_i)|\geq 2^{-N}\}}+\sum_{i\geq 1} |t(Y_i)|\1_{\{|t(Y_i)|<2^{-N}\}}=|X|_{f(t)}+|X|_{h(t)}.
\end{align*}
Since $|X|_t \geq |X|_{f(t)}$ we have (we recall that $K$ may differ at each occurance)
\begin{align} \label{eq:loca1a}\{\sup_{t\in T}|X|_t \geq KS  \}&\subset \{\sup_{t\in T}|X|_{f(t)} \geq K \E\sup_{t\in T}|X|_{f(t)}    \} \cup \ \bigcup_{t\in T}\{|X|_{h(t)} \geq KS  \}.
\end{align}
Markov's inequality implies that
\begin{align*}
\sum_{t\in T} \P(|X|_{h(t)} \geq KS)\leq \sum_{t\in T} \frac{\E |X|_{h(t)}}{KS}=\sum_{t\in T} \frac{\int_{\R^T} |t(x)|\1_{\{|t(x)|<2^{-N}\}} d\nu(x)}{KS}. 
\end{align*}
Since $T$ is finite, by \eqref{nuogr} we have that for sufficiently large  $N$, $\sum_{t\in T} \P(|X|_{h(t)} \geq KS)<\frac{1}{16}$. So by \eqref{eq:loca1a} and the above considerations it is enough to show that the event
\[\{\sup_{t\in T}|X|_{f(t)} \geq K \E \sup_{t\in T}|X|_{f(t)}\}\]
is  $7/16$-small. Thus, without loss of generality, we can assume that 
\begin{equation}\label{eq:ucieciet}
\forall_{t\in T}  \forall_{x\in \R^T} \ \ t(x)=0 \text{ or } |t(x)|\geq 2^{-N}.   
\end{equation}
In particular, this implies that
\begin{equation}\label{nuosz}
\nu\left(x\in \R^T: |t(x)|>0\right) \leq 2^N\int_{\mathbb{R}^T} |t(x)| d\nu(x) \leq 2^N S.
\end{equation}
\noindent {\bf Step 2}  
Let $M:=16S|T|$. For each $t\in T$ denote $B(t):=\{x\in \R^T:\; |t(x)|\geq M \}$. Since for each $t\in T$
\[\{\exists_i : |t(Y_i)|\geq M \}=\{|X|_{\1_{B(t)}}\geq 1\}, \]
we have that
\begin{equation*}
   \{\sup_{t\in T}|X|_t >KS\} \subset \{\sup_{t\in T} \sum_i |t(Y_i)|\1_{|t(Y_i)|\leq M }\geq KS  \}\cup \bigcup_{t\in T}\{|X|_{\1_{B(t)}}\geq 1\}.
\end{equation*}
Observe that
\[
\sum_{t\in T }\P(|X|_{\1_{B(t)}}\geq 1)\leq  \sum_{t\in T } \E |X|_{\1_{B(t)}}\leq  \sum_{t\in T }\E \frac{|X|_t}{M } \leq \sum_{t\in T } \frac{S}{M} \leq \frac{1}{16}.
\]
Thus,  it is enough to show that the following event is $6/16$ small
\[\{\sup_{t\in T} \sum_i |t(Y_i)|\1_{|t(Y_i)|\leq M}\geq KS  \}. \]
\noindent {\bf Step 3.}  
The third step is to replace the measurable functions $t\in T$ (treated as the functions on $\R^T$) with the step functions. Define the following set of sequences, indexed by  $T$
\begin{align*}
    \mathcal{S}:=\{(k(t))_{t\in T}:\forall_{t\in T} 0\leq k(t)\leq M4^N \textrm{ and } \exists_{t\in T} k(t) \neq 0  \}.
\end{align*}
Let
\begin{align*}
\all&:=\{A^N_{(k(t))_{t\in T}} : (k(t))_{t\in T}\in \mathcal{S} \},
\end{align*}  
where
\begin{align*}
   A^N_{(k(t))_{t\in T}}&:=\bigcap_{t\in T}\{x\in \R^T:\; \frac{k(t)-1}{4^N}< |t(x)|\leq  \frac{k(t)}{4^N} \}.   
\end{align*}
If $k(t)=0$ then \eqref{eq:ucieciet} implies
\[\{x\in \R^T:\; \frac{k(t)-1}{4^N}< |t(x)|\leq  \frac{k(t)}{4^N} \}=\{x\in \R^T: x(t)=0 \}.\]
Since the sequence constantly equal to zero does not belong to $\mathcal{S}$, we have
\begin{align}
    \bigcup_{A \in \all} A &= \left( \bigcap_{t\in T}\{x\in \R^T: |t(x)|\leq M\} \right) \setminus \{x\in \R^T: \forall_{t\in T}\ t(x)=0 \}\nonumber \\
    &\subset \bigcup_{t\in T}\{x\in \R^T: 0<|t(x)|\leq M\}  \label{s-10-ineq2}. 
\end{align}
Because $\all$ consists of disjoint sets, by using \eqref{nuosz} we obtain
\begin{equation}\label{sumvosz}
\sum_{A\in \all} \nu(A)\leq \sum_{t\in T} \nu\left( \{x\in \R^T: 0<|t(x)|\leq M\}\right)\leq  2^N|T|S.
\end{equation}
For each $t\in T$ and $A\in \all$ let $k(A,t)$ be the unique integer such that the number $t(A):=k(A,t)/4^N$ satisfies
\begin{equation}\label{przyblizN}
 \forall_{x\in A}\    t(A)-1/4^N< |t(x)|\leq t(A).
\end{equation}
Thus, if we define
\begin{equation}\label{def:poiss-zlicz}
Y(A):=\sum_{i\geq1}\1_{\{Y_i\in A\}},
\end{equation}
by \eqref{s-10-ineq2} and  \eqref{przyblizN} we obtain that
\begin{align*}
    \sup_{t\in T} \sum_i |t(Y_i)|\1_{|t(Y_i)|\leq M} \leq \sup_{t\in T} \sum_{A\in \all}t(A)Y(A).
\end{align*}
As a result
\begin{equation*}\label{ord0}
   \{\sup_{t\in T} \sum_i |t(Y_i)|\1_{|t(Y_i)|\leq M} \geq KS\} \subset \{\sup_{t\in T} \sum_{A\in \all}t(A)Y(A)\geq KS  \},
\end{equation*}
and it is sufficient to show that the latter set is $6/16$ small.

\noindent {\bf Step 4.} Let
\begin{align}
&\ccA_N:=\left\{A: A \in \all,\ \nu(A)>0\right\},& &\ccN_N:=\left\{A:A \in \all,\ \nu(A)=0\right\}. \label{miarzero}
\end{align}
Clearly,
\[ \{\sup_{t\in T} \sum_{A\in \all}t(A)Y(A)\geq KS  \}\subset  \{\sup_{t\in T} \sum_{A\in \ccA_N}t(A)Y(A)\geq K S  \}\cup \bigcup_{A\in \ccN_N}\{|X|_{\1_A}\geq 1 \}. \]
Since
\begin{equation}\label{ord1}
    \sum_{A\in \ccN_N}\P(|X|_{\1_A}\geq 1)=\sum_{A\in \ccN_N} \nu(A)=0,
\end{equation}
it is sufficient to show that $\{\sup_{t\in T} \sum_{A\in \ccA_N}t(A)Y(A)\geq K S  \}$ is $6/16$ small.
Observe that trivially from \eqref{sumvosz}
\begin{equation}\label{sumAosz}
\sum_{A\in \ccA_n} \nu(A)\leq 2^N|T|S.
\end{equation}


\noindent {\bf Step 5.}
Denote $\bar{S}:=\E \sup_{t\in T} \sum_{A\in \ccA_N}t(A)Y(A)$. Using in the first inequality \eqref{przyblizN} and \eqref{sumAosz} in the second  we get that
\begin{align*}
S&=\E \sup_{t\in T}\sum_{i\geq 1} |t(Y_i)|\geq \E \sup_{t\in T}\sum_{A\in \ccA_n} (t(A)-\frac{1}{4^N})Y(A)=\bar{S}-\sum_{A\in \ccA_n} \frac{1}{4^N}\E Y(A)\\
&=\bar{S}-\frac{1}{4^N}\sum_{A\in\ccA_N}\nu(A)\geq \bar{S}-\frac{2^N|T|S}{4^N}\geq \frac{\bar{S}}{2}
\end{align*}
for $N$ large enough. So, 
\[\left\{\sup_{t\in T} \sum_{A\in \ccA_n} t(A)Y(A)\geq KS  \right\} \subset \left\{\sup_{t\in T} \sum_{A\in \ccA_n} t(A)Y(A)\geq K\bar{S}  \right\} \] 
and it is enough to prove that the latter event  is $6/16$-small (note that the costant $K$ is now different).

\noindent {\bf Step 6.}  The next step is to divide each $A\in \ccA_N$ into subsets of equal measure and a "little" reminder. Let (since $|\ccA_N|<\infty$ the inequality follows)
\begin{equation}\label{eq:wybp}
    p:=\frac{\min\left(1,S,\inf_{A\in \ccA_n} \nu (A)\right)}{16\cdot 2^N S |T|} \in (0,1).
\end{equation}
For each $A\in \ccA_N$ let $m(A):=\lfloor \nu(A)/p \rfloor$. Fix $A\in \ccA_N$.  Since $\nu$ is atomless there exist $\nu$-measurable disjoint sets $A_1, A_2,\ldots A_{m(A)}\subset A$ such that $\nu(A_k)=p$. Let also $A_{\ast}:=A\backslash \bigcup^{m(A)}_{k=1}A_k$.   Obviously 
\begin{equation}\label{tuv}
    \nu(A_{\ast})\leq p \leq \frac{\nu(A)}{16\cdot 2^NS|T|}.
\end{equation}
For any $A\in \ccA_N$  the sets $(A_i)_{i\leq m(A)}$ are disjoint. Thus,
\begin{equation}\label{eq:zlicz}
 \forall_{A\in \ccA_n}\ \sum_{i=1}^{m(A)}\1_{Y(A_i)>0}  \leq     \sum_{i=1}^{m(A)}Y(A_i) \leq Y(A).
\end{equation}
We define
\[\rod:=\{A_k :\ A\in \ccA_N,k\leq m(A) \}, \]
and for each $B\in \rod$
\begin{align}\label{eq:deftb}
t(B):=t(A) \textrm{ where } A\in \ccA_N \textrm{ is the unique set s.t. }  B\subset A.
\end{align}
Since the constructed $t(A),t(B)$ are nonnegative and using $\eqref{eq:zlicz}$ we obtain 
\begin{align*}
\bar{S}=\E \sup_{t\in T} \sum_{A\in \ccA_N}t(A)Y(A) \geq \E \sup_{t\in T} \sum_{B\in \rod}t(B)\1_{Y(B)>0}:=S'.    
\end{align*}
By the above
\begin{align*}
\left\{\sup_{t\in T} \sum_{A\in \ccA_N}t(A)Y(A)\geq K\bar{S}  \right\} \subset 
  \left\{\sup_{t\in T} \sum_{B\in \rod} t(B)Y(B)\geq KS'  \right\}\cup \bigcup_{A\in \ccA_N}\{|X|_{\1_{A_{\ast}}}\geq 1\}.
\end{align*}

From \eqref{tuv} and \eqref{sumAosz} we get that
\begin{align*}
\sum_{A\in \ccA_N} \P( |X|_{\1_{A_{\ast}}}\geq 1) \leq \sum_{A\in \ccA_N} \E |X|_{\1_{A_{\ast}}}=\sum_{A\in \ccA_N} \nu(A_{\ast})\leq \sum_{A\in \ccA_N} \frac{\nu(A)}{16\cdot 2^N S|T|}\leq \frac{1}{16}.
\end{align*}
So it is sufficient to show that the event 
\[\left\{\sup_{t\in T} \sum_{B\in \rod} t(B)Y(B)\geq KS'  \right\} \]
is $5/16$-small.

\noindent  {\bf Step 7.}
This step shows that the random variables $Y(B)$ can be replaced by the independent Bernoulli variables. Let $n:=|\rod|$. Formulas \eqref{sumAosz}, \eqref{eq:wybp} (and since for all $B\in \rod$ we have $\nu(B)=p$) imply
\begin{equation}\label{eq:pokemon}
np^2= p\sum_{B\in\rod}\nu(B)\leq p\sum_{A\in \ccA_N} \nu(A)\leq \frac{1}{16}.    
\end{equation}
 Clearly,
\[ \{Y(B)>1 \}=\{|X|_{\1_{B}}\geq 2 \}. \]
So, we immediately obtain that
\[
\left\{\sup_{t\in T}  \sum_{B\in \rod}   t(B)Y(B)\geq KS' \right\} \subset \left\{\sup_{t\in T}  \sum_{B\in \rod}  t(B)\1_{\{Y(B)>0\}} \geq KS' \right\} \cup \bigcup_{B\in \rod} \{|X|_{\1_{B}}\geq 2 \}.
\]
Since $(Y_i)_{i \geq 1}$ is the Poisson point process the random variables $(Y(B))_{B \in \rod}$ are independent (we will need this property later) Poisson r.v.'s with the parameter $\nu(B)=p$. Thus (recall \eqref{eq:pokemon}),
\begin{align*}
\sum_{B\in \rod}  \P(|X|_{\1_{B}}\geq 2)&=\sum_{B\in \rod}  \P(Y(B)>1)=n(1-e^{-p}(1+p))\leq np^2\leq \frac{1}{16}.
\end{align*}
We have reduced our problem to the task of proving that the following event is $1/4$-small
\[ \left\{\sup_{t\in T}  \sum_{B\in \rod}  t(B)\1_{\{Y(B)>0\}}\geq KS'  \right\}.\]
\noindent {\bf Final step.} 
The random variables $(Y(B))_{B\in \rod}$ are i.i.d. Poisson$(p)$ r.v.'s (see the previous step). Thus, Theorem \ref{main1} ensures that the event 
\[\left\{\sup_{t\in T}  \sum_{B\in \rod}   t(B)\1_{\{Y(B)>0\}} \geq KS'\right\}\]
is $(1-e^{-p})$-small. Since $p\geq 1-e^{-p} \geq p/e$ (recall \eqref{eq:wybp}) by Remark \ref{rem-mniejszy-zbior} we can argue that for sufficiently large, numerical constant $K$ this event is $(2ep)$-small. This  means that
there exists a family $\ccF$ of subsets of the set $\rod$   such that
\[
\left\{\sup_{t\in T}  \sum_{B\in \rod}   t(B)\1_{\{Y(B)>0\}} \geq KS'\right\}\subset \bigcup_{F\in \ccF}\bigcap_{B\in F}\{Y(B)>0\}
\]
and
\begin{equation}\label{eq:koniecsum}
\sum_{F\in \ccF}(2ep)^{|F|}\leq \frac{1}{2}.    
\end{equation}
Now, for $F\in \ccF$ define
\[g_{F}(x):=\1_{\bigcup_{B\in F}}(x).\]
  Note that   
 \[
\bigcap_{B\in F}\{Y(B)>0\}\subset \left\{\sum_{B\in F }Y(B)\geq |F|\right\}=\{|X|_{g_{F}}\geq |F|\},
 \]
so
\[\cbr{\sup_{t\in T}  \sum_{B\in \rod}   t(B)\1_{\{Y(B)>0\}} \geq KS'}\subset \bigcup_{F\in \ccF} \cbr{|X|_{g_{F}}\geq |F|}.\]
Since  $\sum_{B\in F }Y(B)\sim Poiss(p|F|)$ Chernoff's inequality implies that (we recall that $p<1$ so that $\ln p<0$)
\begin{align*}
\P(|X|_{g_{F}}\geq |F|)&=\P\left(\sum_{B\in F }Y(B)\geq |F|\right)\leq e^{|F| \ln p}\E e^{-\ln p \sum_{B\in F }Y(B) }\\
&=e^{|F| \ln p} e^{p|F|(\frac{1}{p}-1)}
\leq (ep)^{|F|}.
\end{align*}
By recalling \eqref{eq:koniecsum} we obtain 
\[\sum_{F\in \ccF}\P(|X|_{g_{F}}\geq |F|)\leq \sum_{F\in \ccF} (ep)^{|F|}\leq \frac{1}{4}.\]
Thus, we have found the witness that the event  
$$\cbr{\sup_{t\in T}  \sum_{B\in \rod}   t(B)\1_{\{Y(B)>0\}}\geq KS'  }$$ is $1/4$-small.
\end{proof}

\section{Positive empirical processes} 
The proof of Theorem \ref{twemp} is similar to the one in the previous section. We want to partition $E$ into a sequence of appropriately small sets with respect to the measure $\mu$, argue that it is enough to investigate only one representative of each of the partition sets, and therefore reduce the problem to the selector case. The main difference is that the analogs of indicators from Step $7$ are now not independent. To deal with this fact, we first prove a version of Theorem \ref{main1}. \\
\noindent 
Let us use the notation $\binom{[n]}{m}$ for a class of all $m$-element subsets of $[n]$. Suppose that $Y\in\binom{[n]}{m}$ is chosen uniformly. By $\P_Y$ and $\E_Y$ we denote the probability and the expectation with respect to $Y$.
\noindent We aim to prove the following conditional version of Theorem \ref{main2}.  
\begin{theo}\label{master}
Let $\ccF$ be a family of subsets of $[n]$ which is not $p$-small. Suppose that for each $I\in\ccF$ we are given coefficients $(\mu_I(i))_{i\in I}$ with $\mu_I(i)\geq0$ and  $\sum_{i\in I}\mu_I(i)\geq 1.$
Let $Y$ be uniformly distributed on $\binom{[n]}{m}$ where $m\geq 9pn$. Then 
$$\E_Y\sup_{I\in\ccF}\sum_{i\in I\cap Y}\mu_I(i)\geq\frac{1}{10}.$$
\end{theo}
\begin{proof}
Let $X$ be given by \eqref{randomset}. Then $X$ conditionally on $\{|X|=m\}$ has the same distribution as $Y$ (uniform). So by Lemma \ref{zly}
\[\P(Y \textit{is 1/2-bad})\leq \sum_{t=1}^n \left(4\frac{np}{m} \right)^t \leq \sum_{t\geq1} \left(\frac{4}{9}\right)^t \leq  \frac{4}{5}. \]
Thus
\[ \E_Y\sup_{I\in\ccF}\sum_{i\in I\cap Y}\mu_I(i)\geq \frac{1}{2} \P(Y \textit{ is not 1/2-bad})\geq \frac{1}{2} \left(1-\frac{4}{5} \right).\]
\end{proof}

\begin{rema}
Theorems \ref{main1} and \ref{master} are in the sense exposition of the same result from two perspectives. Theorem \ref{master} emphasizes the necessary relation between $m$ and $n$ for providing a $p$-small cover which we need in the application for empirical processes.
\end{rema}

\begin{theo}\label{malarodzina}
The family
$$\ccF=\cbr{I\in\binom{[n]}{m}: \sup_{t\in T}\sum_{i\in I}t_i\geq 11\E_Y\sup_{t\in T}\sum_{i\in Y}t_i}$$
is $\frac{m}{9n}$-small.
\end{theo}
\begin{proof}
Assume for contradiction that $\ccF$ is not $\frac{m}{9n}$-small. Let $S:=\E_Y\sup_{t\in T}\sum_{i\in Y}t_i$. For any $I\in\ccF$ we can find $t\in T$ such that for $\mu_I(i):=\frac{t_i}{11S}$, we have $\sum_{i\in I}\mu_I(i)\geq 1$, thus
$$\E_Y\sup_{I\in\ccF}\sum_{i\in I\cap Y}\mu_I(i)\leq\E_Y\sup_{t\in T}\sum_{i\in Y}\frac{t_i}{11S}=\frac{1}{11}<\frac{1}{10},$$
which contradicts Theorem \ref{master}.
\end{proof}
\begin{lema}\label{porsup}
Let $Y$ and $Y'$ be uniformly distributed over $\binom{[n]}{m}$ and $\binom{[n]}{km}$ respectively, where $k$ is a natural number such that $km\leq n$. Then
\[ \E_{Y'}\sup_{t\in T} \sum_{i\in Y'} t_i \leq k \E_{Y} \sup_{t\in T} \sum_{i\in Y} t_i.\]
\end{lema}
\begin{proof}
 Let $\pi$ be a random permutation uniformly distributed over all permutations of the set $[km]$. Let $y_1<y_2<
\ldots<y_{km} $ be the elements of $Y'$. Then
\begin{align*}
S'=\E_{\pi, Y'}\sup_{t\in T} \sum_{i=1}^{km}  t_{y_{\pi(i)}} \leq \sum_{l=0}^{k-1} \E_{\pi, Y}\sup_{t\in T} \sum_{i=lm+1}^{(l+1)m}  t_{y_{\pi(i)}}=kS.
\end{align*}
\end{proof}
\begin{prop}\label{remark1}
 Let $C\geq 1$, $Y$ be uniformly distributed over $\binom{[n]}{m}$ and
 \[S=\E_Y \sup_{t\in T} \sum_{i\in Y} t_i.\] 
 Then the family
    \[ \ccF=\cbr{I\in \binom{[n]}{m}: \sup_{t\in T} \sum_{i\in I} t_i \geq 100CS}\]
    is $\frac{Cm}{n}$-small.
\end{prop}
\begin{proof}
 For any $I\in \binom{[n]}{m}$ we pick  $t^I\in T$ such that
 \[ \sup_{t\in T} \sum_{i\in I} t_i = \sum_{i\in I} t^I_i.\]
 Then
 \[\ccF = \cbr{I\in \binom{[n]}{m}: \sum_{i\in I} t^I_i \geq 100CS}. \]
 Assume that $\ccF$ is not $\frac{Cm}{n}$-small. Let $k:=\lceil 9C \rceil$. If $km>n$ then by Lemma \ref{porsup} family $\ccF$ is empty which is a contradiction. Therefore we may assume that $km\leq n$. Take $Y'$ uniformly distributed over $\binom{[n]}{km}$. By Theorem \ref{master}
 \begin{equation}\label{sprzecz1}
 \E_{Y'} \sup_{I\in \ccF} \sum_{i\in Y'} t^I_i \geq 10CS.
 \end{equation}
 On the other hand, by Lemma \ref{porsup} we have that
 \begin{equation}\label{sprecz2}
    \E_{Y'} \sup_{I\in \ccF} \sum_{i\in Y'} t^I_i\leq \E_{Y'} \sup_{t\in T} \sum_{i\in Y'}t_i \leq  k S.
 \end{equation}
Trivially \eqref{sprzecz1} and \eqref{sprecz2} give a contradiction so  $\ccF$ is $\frac{Cm}{n}$-small.
 
\end{proof}

\begin{proof}[Proof of Theorem \ref{twemp}]
We proceed as in the case of infinitely divisible processes gradually simplifying the functions $t\in T$. The proof is very similar but less technical. This is due to the fact that we are now working with the probabilistic measure $\mu$ instead of the (potentially) infinite measure $\nu$. For example, Step 1 in the proof of Theorem \ref{twpod} is now superfluous. Only the last step requires a different argument. For this reason, we decided to refer to the proof of Theorem \ref{twpod} in a few places, instead of repeating the technical arguments.\\
\noindent Throughout the proof we use the analogous notion of $\delta$-smallness from Definition \ref{deltasmall}, where we replace $|X|_g$ with $\frac{1}{d} \sum_{i=1}^{d} g(X_i)$. Let $S:=\E\sup_{t\in T}\sum_{i=1}^dt(X_i)$. and $M>0$ be a large number that will be precised later. In particular, we may assume that (since $T$ is finite)
\begin{equation}\label{duzeM}
    \sum_{t\in T} \mu(\{|t|\geq M\})\leq \frac{1}{16d}.
\end{equation}


\noindent {\bf Step 1.} 
Fix $N$ large enough. By repeating the argument from Step 3 of the proof of Theorem \ref{twpod} we construct $\all$ a family of disjoint subsets of $E$, together with coefficients $(t(A))_{t\in T,A\in \all}$ such that
\begin{align}\label{wysuma}
   \bigcup_{A\in \all} A =\bigcap_{t\in T}\{ x\in E:|t(x)|\leq M\}
   \end{align}
and for any $A\in\all$, $x\in A$,   
   \begin{align}\label{przybliz}
    t(A)-4^{-N}\leq t(x)\leq t(A).
\end{align}
Notice the difference between \eqref{wysuma} and \eqref{s-10-ineq2}. The reason is that $\mu$ is finite so we don't need to take care of the behavior of the measure around $0$. Let $X(A):=\sum_{i=1}^d\1_{\{X_i\in A\}}$.
Then \eqref{wysuma} and \eqref{przybliz} imply that
\[ \left\{\sup_{t\in T}\sum_{i=1}^d t(X_i) >KS \right\} \subset \left\{\sup_{t\in T}\sum_{A\in \all} t(A)X(A) >KS \right\} \cup \bigcup_{t\in T} \left\{\sum_{i=1}^d \1_{|t(X_i)|\geq M} \geq 1 \right\}. \]
Using Bernoulli's inequality and \eqref{duzeM}
\[\sum_{t\in T} \P\left(\cbr{\sum_{i=1}^d \1_{|t(X_i)|\geq M} \geq 1}\right)=\sum_{t\in T}\left(1-\mu(\{|t|\leq M \})^d\right)\leq d\sum_{t\in T} \mu(\{|t|>M \})\leq \frac{1}{16}. \]
So it is sufficient to show that the set $\{\sup_{t\in T}\sum_{A\in \all} t(A)X(A) >KS \}$ is $7/16$-small.

\noindent {\bf Step 2.} 
Let $\ccA_N:=\{A:A\in \all,\ \mu(A)>0$\} and $\ccN_N:=\{A: A\in \all,\ \mu(A)=0\}$. Then we get that
\begin{align*}
  \left\{\sup_{t\in T}\sum_{A\in \all} t(A)X(A) >KS \right\}\subset \Bigg( \left\{\sup_{t\in T}\sum_{A\in \ccA_N} t(A)X(A) >KS \right\}&\\
 \cup \bigcup_{A\in \ccN_N}\left\{ \sum_{i=1}^d \1_A(X_i)\geq 1\right\}&\Bigg).   
\end{align*}
By the definition of $\ccN_N$ we have $\sum_{A\in \ccN_N}\P(\sum_{i=1}^d \1_A(X_i)\geq 1)=0$. Also, by analogous argument as in Step 5 of the proof of Theorem \ref{twpod},
\[S \geq \frac{\E \sup_{t\in T} \sum_{A\in \ccA_N} t(A)X(A)  }{2}=:\frac{\bar{S}}{2}.\]
Thus, it is enough to show that the following set is $7/16$-small
\[\left\{\sup_{t\in T}\sum_{A\in \ccA_N} t(A)X(A) >K\bar{S} \right\}.\]
\noindent {\bf Step 3.} Let 
\begin{align}\label{eq:p2}
    p:=\frac{\min(1,\inf_{A\in \ccA_N }\mu(A))}{16d^2}>0.
\end{align} 
We can divide each $A\in \ccA_N$ into disjoint sets $A_1,\ldots,A_{m(A)},A_*$, where $m(A)=\lfloor \mu(A)/p\rfloor$, $\mu(A_1)=\ldots=\mu(A_{m(A)})=p$ and $\mu(A_*)<p$ (we recall that $\mu$ is atomless). Let (as before)
\[\rod:=\{A_k : A\in \ccA_N, k\leq m(A) \}. \]
From the definition of $p$ (and since $\ccA_N$ consists of disjoint sets)
\begin{equation}\label{dofin1}
 \sum_{A\in \ccA_N}\P(\sum_{i=1}^d \1_{A_*}(X_i) \geq 1 ) \leq d \sum_{A\in \ccA_N} \mu(A_*) \leq d \sum_{A\in \ccA_N} \frac{\mu(A)}{16d^2} \leq  \frac{1}{16}.   
\end{equation}
Similar to Step 6 of the proof of Theorem \ref{twpod}, we show that the sets $A^*$ are negligible, so it is sufficient to show that the event 
\begin{equation}\label{tenzbior}
   \cbr{\sup_{t\in T}\sum_{B\in\rod}t(B)X(B)\geq K\bar{S} } 
\end{equation}
is $5/16$-small, where $(t(B))_{B\in \rod}$ are defined as in \eqref{eq:deftb}.\\
\noindent {\bf Step 4.} In this step we replace the variables $(X(B))_{B\in \rod}$ by $0-1$ variables (but not independent ones). By the construction, the sets $(B)_{B\in \rod}$ are disjoint and for any $B\in \rod$ $\mu(B)=p$. So, if $n=|\rod|$ we have
\begin{align}\label{eq:np}
 np=\sum_{B\in \rod} \mu(B) \leq 1.   
\end{align}
Now, observe that
\[\{X(B)>1\}=\cbr{\sum_{i=1}^d \1_{B}(X_i)\geq 2 },\]
so (as before) for
$$\ccE:=\cbr{\sup_{t\in T}  \sum_{B\in\rod}  t(B)X(B)\geq K\bar{S}}\cap\bigcap_{B\in \rod}\{X(B)\in\{0,1\}\} $$
we have
$$\cbr{\sup_{t\in T}  \sum_{B\in\rod}t(B)X(B)\geq K\bar{S}}\subset\ccE\cup \bigcup_{B\in \rod} \cbr{\sum_{i=1}^d \1_{B}(X_i)\geq 2}.$$
Since $(1-p)^d\geq 1-pd$, we obtain
\begin{equation}\label{dofin2}
  \sum_{B\in\rod}\P\left(\sum_{i=1}^d \1_{B}(X_i)\geq 2 \right)=n(1-(1-p)^d-dp(1-p)^{d-1})\leq np^2d^2\leq \frac{1}{16}, 
\end{equation}
where the last inequality follows from \eqref{eq:np} and \eqref{eq:p2}. By the same argument as in Step 6 of the proof of Theorem \ref{twpod}
\begin{align*}
 \bar{S}=\E\sup_{t\in T}\sum_{A\in\ccA_N}t(A)X(A)\geq \E\sup_{t\in T}\sum_{B\in\rod}t(B)X(B) =:S'.
\end{align*}
So it is enough to show that the following event is $1/4$-small
\begin{equation}\label{defeprim}
\ccE':=\cbr{\sup_{t\in T}  \sum_{B\in\rod}  t(B)X(B)\geq KS'}\cap\bigcap_{B\in \rod}\{X(B)\in\{0,1\}\}.     
\end{equation}
\noindent \textbf{Final Step.}
Recall that $\binom{\rod}{d}$ is a family of $d$ element subsets of $\rod$. Let $\ccY$ be uniformly chosen from $\binom{\rod}{d}$ and let $\ccX\in 2^\rod$ be a random set defined as $\ccX:=\{B\in \rod: X(B)\geq 1\}$. We claim that 
\begin{equation}\label{X-Y}    
\frac{8}{7} \E \sup_{t\in T} \sum_{B \in \rod} t(B)X(B)=\frac{8}{7} \E \sup_{t\in T} \sum_{B \in \ccX} t(B)X(B)\geq \E\sup_{t\in T} \sum_{B\in \ccY} t(B). 
\end{equation}
By \eqref{dofin1}, \eqref{dofin2} and the union bound
\begin{multline*}
    \P\left(|\ccX|\neq d \textrm{ or } \exists_{B\in \rod} X(B)\geq 2\right)\\
    \leq \sum_{A\in \ccA_n} \P\left(\sum_i \1_{A_\ast}(X_i) \geq 1 \right) +\sum_{B\in \rod} \P\left(\sum_i \1_{B}(X_i) \geq 2\right)\leq \frac{1}{8}.
\end{multline*}
Thus,
\[\E \sup_{t\in T} \sum_{B \in \ccX} t(B)X(B)\geq \frac{7}{8} \E \left( \sup_{t\in T} \sum_{B \in \ccX} t(B)X(B) \mid |\ccX|=d,\ \forall_{B\in \rod} X(B)\in \{0,1\}\right).   \]
Now, if we take any $\ccI=\{B_1,\ldots,B_d\}\in \binom{\rod}{d}$ and any other $\ccJ\in \binom{\rod}{d}$ we see that
\[\P(\ccX=\ccI)=\sum_{\pi\in \mathrm{Perm}([d])} \P(X_{\pi(1)}\in B_1,\ldots,X_{\pi(d)}\in B_d )=d!p^d=\P(\ccX=\ccJ).  \]
Thus $\ccX$ conditioned on $\{|\ccX|=d,\ X(B)\in \{0,1\} : B\in \rod \}$ has the same distribution as $\ccY$ which concludes the argument for \eqref{X-Y}.\\
\noindent Now, by Proposition \ref{remark1} there exists a $\frac{4ed}{n}$-small family $\ccG$  consisting of the sets $G=\{B_1,B_2,\ldots,B_{|G|}: B_i\in \rod \}$ such that
\begin{multline*}
  \cbr{\ccI\in \binom{\rod}{d} : \sup_{t\in T} \sum_{B\in \ccI}t(B) \geq 400e \E \sup_{t\in T} \sum_{B \in \rod} t(B)X(B) }\\
  \subset \cbr{\ccI\in \binom{\rod}{d} : \sup_{t\in T} \sum_{B\in \ccI}t(B) \geq 400e \E \sup_{t\in T} \sum_{B\in \rod} t(B)X(B)}
  \subset \bigcup_{G\in \ccG} \bigcap_{B\in G} 
 \{\ccI : B\in \ccI\}  
\end{multline*}
and $\sum_{G\in\ccG}\left(\frac{4ed}{n}\right)^{|G|}\leq\frac{1}{2}$.
Then clearly (remembering the definition \eqref{defeprim} of $\ccE'$)
\[\ccE'\subset  \bigcup_{G\in \ccG} \{X(B)=1, B\in G \} \cap \{X(B)\in\{0,1\};B\in \rod \}. \]
We define $g_G(x):=\sum_{B\in G} \1_B(x)$ and see that
\[\{X(B)=1, B\in G \} \cap \{X(B)\in\{0,1\};\; B\in \rod \}\subset \left\{\sum_{i=1}^{d} g_G(X_i)\geq |G| \right\}.\]
Thus, it is enough to show that
\[\sum_{G\in \ccG}\P\left(\sum_{i=1}^{d} g_G(X_i)\geq |G|\right)\leq 1/4. \]
The variables $(g_G(X_i))_{i\leq d}$ are independent Bernoulli random variables with a probability of success equal to $|G|p$. Hence,
\begin{align*}
\P\left(\sum_{i\leq d} g_G(X_i)\geq |G|\right)\leq \binom{d}{|G|}(p|G|)^{|G|}\leq (edp)^{|G|},
\end{align*}
so the result follows for sufficiently small $p$.
\end{proof}

\section*{Appendix}
The aim of this section is to present how the function $F$ of Definition \ref{jotyiepsilony} can be used in a more general setting. For the sake of clarity, and since the proof steps are conceptually not far from those presented in Section $3$, we will not go into every detail of the argument.\\
\noindent Consider independent random variables $X_1,X_2,\ldots,X_d$ distributed in the set $[n]$ 
according to the same probability distribution, that is, $\P(X_i=j)=p(j)>0$, $j\in [n]$. Let $T\subset \R_{+}^d$
and let $f:[n]\ra \R_{+}$. We assume that $T$ is permutationally invariant i.e., for any permutation $\sigma:[d]\ra [d]$  and any $t\in T$ we have $(t_{\sigma(i)})^d_{i=1}\in T$.
We are interested in
\[
S(T):=\E\sup_{t\in T}\sum^d_{i=1}t_i f(X_i)=\E \sum^d_{i=1}t_i^X f(X_i),
\]
where for $x\in[n]^d$, $t^x\in T$ is defined as the point at which $\sup_{t\in T}\sum^d_{i=1}t_i f(x_i)$ is attained and $t^X$ is its random counterpart. Since $T$ is permutationally invariant, we see that for any permutation $\sigma$ and $X_{\sigma}=(X_{\sigma(i)})^d_{i=1}$,
we can choose $t^{X_{\sigma}}=t^X(\sigma)$,  where $t^X_i(\sigma)=t^X_{\sigma(i)}$ for $i\in [d]$.
\smallskip

\noindent
For $x\in[n]^d$ and $1\leq j\leq n$ let $s^x_d(j):=|\{i\in[d]: x_i=j\}|$ and let us call $s^x_d:=(s^x_d(j))_{j=1}^n$ the multiset generated by $x$. In general, we refer to an element of $[d]^n=\{0, 1,\dots, d\}^n$ as the multiset. Let $\ccM_d:=\{s=(s(j))_{j=1}^n\in\{0, 1,\dots, d\}^n: \sum^n_{j=1}s(j)=d\}$, which we refer to as the collection of multisets of length $d$. Observe that the multiset $s^X_d$ generated by $X=(X_1,\dots, X_d)$ is the random element of $\ccM_d$ distributed according to 
\begin{equation}\label{defmulti}
\P(s^X_d=s)=d!\prod_{j=1}^n\frac{p(j)^{s(j)}}{s(j)!}.
\end{equation}
We will need the following coefficients
\[
\mu^{s^x_d}_j:=\frac{1}{d! s^x_d(j)}\sum_{\sigma} \sum^d_{i=1} t^x_{\sigma(i)} f(j) \1_{x_{\sigma(i)}=j},\ j=1,\ldots,n.
\]
The following simple observation is a starting point for our result.
\begin{lema}
It holds that
\[
S(T)=\E \sum^n_{j=1}\mu^{s^X_d}_j s^X_d(j).
\]
\end{lema}
\begin{proof}
We have 
\begin{align*}
&S(T)=
\E \sum^d_{i=1}t^X f(X_i)= \E \sum^n_{j=1} \frac{1}{s^X_d(j)} \sum^d_{i=1} t^X_{i} f(j) \1_{X_i=j} s^X_d(j) \\
& =\E \sum^n_{j=1} \frac{1}{d!s^X_d(j)}\sum_{\sigma}\sum^d_{i=1} t^X_{\sigma(i)}f(j) \1_{X_{\sigma(i)}=j} s^X_d(j)\\
&=\E \sum^n_{j=1}\mu^{s^X_d}_j s^X_d(j).
\end{align*}    
\end{proof}
\begin{defi}
Let $\ccH, \ccG \subset [n]^d $ be families of multisets. We say that $\ccG$ is a cover of $\ccH$ if for every multiset $s\in\ccH$ there exists $w\in\ccG$ such that for $1\leq j\leq n$, $w(j)\leq s(j)$.
\end{defi}
\begin{defi}\label{A(w)}
    We say that the family $\ccH$ of multisets is $\delta$- small if there exists a cover $\ccG$ of $\ccH$, which satisfies
\[
\sum_{w\in \ccG}A(w) \leq \delta,\;\;\mbox{where}\;\; A(w):=\frac{\prod^n_{j=1}(edp(j))^{w(j)}}{\prod^n_{j=1}w(j)!}.
\] 
\end{defi}
\begin{theo}\label{tw.per}
 There exists a constant $K>0$ such that the family \[
\tilde{\ccF}=\cbr{s^{x}_d\in\ccM_d,\;\; \sum^n_{j=1}\mu^{s^x_d}_j s^x_d(j) \geq KS(T)}\]  is $1/2$-small.   
\end{theo}
\noindent The following result shows that $A(w)$ is a natural quantity, which leads to the characterization of the small cover of $\tilde{\ccF}$ in terms of simple witnessing events, which is in line with the previous results (Theorem \ref{twpod} and Theorem \ref{twemp}). See also the discussion on p.1310 in \cite{Park}.
\begin{theo}
    
    Let $\ccG$ be the cover of $\tilde{\ccF}$ from Theorem \ref{tw.per}. For fixed $w\in\ccG$, consider the event $$E=\{s^X_d\in\ccM_d: w(j)\leq s^X_d(j), 1\leq j\leq n\}.$$
    Then there exists a function $g_w$ and a number $u_w$ such that $E\subset\{\sum_{i=1}^d g_w(X_i)\geq u_w\}$, where $X=(X_1,\dots, X_d)$ generates $s^{X}_d\in E$. Moreover,
    $$\sum_{w\in\ccG}\P(\sum_{i=1}^d g_w(X_i)\geq u_w)\leq\frac{1}{2}.$$
  \end{theo}
\begin{proof}
Suppose that $w\in\ccG$. Define $$g_w(j)=\log(1+\frac{w(j)}{dp(j)}) \;\;\mbox{and}\;\; u_w=\sum_{j=1}^n w(j)\log(1+\frac{w(j)}{dp(j)}).$$ Then $X$ generating $s_d^X\in E$ satisfies
\begin{align*}
    \sum_{i=1}^d g_w(X_i)&=\sum_{j=1}^ns_d^X(j)g_w(j)
    =\sum_{j=1}^n s^X_d(j)\log\rbr{1+\frac{w(j)}{dp(j)}}\\
    &\geq\sum_{j=1}^n w(j)\log\rbr{1+\frac{w(j)}{dp(j)}}=u_w,
\end{align*}
so $E\subset\{\sum_{i=1}^d g_w(X_i)\geq u_w\}$. Now, we show that 
$$\P(\sum_{i=1}^d g_w(X_i)\geq u_w)\leq A(w).$$
Suppose that $\sum_{j=1}^n w(j)=t$. By the definition of the cover we have that $t\leq d$. By Markov's inequality, we obtain 
\begin{align*}
& \P(\sum^d_{i=1}g_w(X_i)\geq u_w) =\P(\exp(\sum^d_{i=1}g_w(X_i))\geq e^{u_w})\leq [\E \exp(g_w(X_1))]^d e^{-u_w}.
\end{align*}
Now,
\begin{align*}
    e^{-u_w}&=e^{-\sum_{j=1}^n w(j)\log\rbr{1+\frac{w(j)}{dp(j)}}}=\prod_{j=1}^n e^{-w(j)\log\rbr{1+\frac{w(j)}{dp(j)}}}\\
    &=\prod_{j=1}^n\rbr{\frac{dp(j)}{dp(j)+w(j)}}^{w(j)}\leq\prod_{j=1}^n\rbr{\frac{dp(j)}{w(j)}}^{w(j)}
\end{align*}
and
\begin{align*}
\E\exp(g_w(X_1))&=\E\exp\rbr{\log\rbr{1+\frac{w(X_1)}{dp(j)}}}=\E\rbr{1+\frac{w(X_1)}{dp(j)}}\\
&=1+\sum_{j=1}^n\frac{w(j)}{dp(j)}p(j)=1+\frac{t}{d}.
\end{align*}
Putting the above together gives
\begin{align*}
\P(\sum^d_{i=1}g_w(X_i)\geq u_w)&\leq (1+\frac{t}{d})^d\prod^{n}_{j=1}\rbr{\frac{dp(j)}{w(j)}}^{w(j)}\leq e^t \prod^{n}_{j=1}\rbr{\frac{dp(j)}{w(j)}}^{w(j)}\\
&=\prod^{n}_{j=1}\rbr{\frac{edp(j)}{w(j)}}^{w(j)}\leq A(w),
\end{align*}
since $w(j)^{w(j)}\geq w(j)!$.  So, by Theorem \ref{tw.per},
$$\sum_{w\in\ccG}\P(\sum_{i=1}^d g_w(X_i)\geq u_w)\leq\frac{1}{2}.$$
\end{proof}
\noindent To shorten the notation, we write $S_d$ for $s^X_d$. Consider $S_{Cd}$ for some $C>0$, which we define as the multiset of length $Cd$ generated by $Y^1, \dots, Y^C$, independent copies of $X$. Clearly, $S_{Cd}$ is distributed on $\ccM_{Cd}$ according to
\[\P(S_{Cd}=s)=(Cd)!\prod_{j=1}^n\frac{p(j)^{s(j)}}{s(j)!}.\]
We aim to prove the following.
\begin{theo}\label{permu}
    Suppose that $\ccF\subset\ccM_{d}$ is not $\frac{1}{2}$- small and that for every $s\in\ccF$ we are given coefficients $\mu^s_j\geq 0$ satisfying $\sum_{j=1}^n \mu_j^s s(j)\geq 1$. Then, there exists $C>0$ such that
    $$\E\sup_{s\in\ccF}\sum_{j=1}^n\mu_j^s S_{Cd}(j)\geq\frac{1}{4},$$
    where $S_{Cd}$ is distributed as above.
\end{theo}

\begin{proof}[Proof of Theorem \ref{tw.per}]
Suppose for contradiction that $\tilde{\ccF}$ is not $1/2$- small.
Let $$\tilde{S}(T)=\E\sup_{s_d^{x}\in\tilde{\ccF}}\sum_{j=1}^n\mu_j^{s^x_d}S_{Cd}(j).$$ Consider coefficients $\tilde{\mu}_j^x:=\frac{\mu_j^{s_d^x}}{KS(T)}$. So, for $s_d^x\in\tilde{\ccF}$, we have $\sum_{j=1}^n \tilde{\mu}_j^x s_d^x(j)\geq 1$. Then, by the argument similar to that of Lemma \ref{porsup}, and by Theorem \ref{permu}, we have 
$$CS(T)\geq \tilde{S}(T)=KS(T)\E\sup_{s_d^x\in\tilde{\ccF}}\sum_{j=1}^n\tilde{\mu}_j^xS_{Cd}(j)\geq\frac{1}{4}KS(T),$$
which is a contradiction for $K>4C$.
\end{proof}

\noindent From now on, we fix the family $\ccF$ from the Theorem \ref{permu}. We need the following definition.
\begin{defi}
Fix $C \in \N$  and  $c\in (0,1)$. We say that $s\in \ccM_{Cd}$ is bad if 
\[
\sup_{r\in \ccF}\sum^n_{j=1}\mu^r(j)s(j) < c.
\]
\end{defi}
\noindent The proof of Theorem \ref{permu} is based on the following result which is analogue of the pivotal Lemma \ref{zly}.
\begin{lema}\label{key2}
Let $c=\frac{1}{2}$. There exists $C>0$ such that
    $$\P(S_{Cd}\;\;\mbox{is bad})\leq\frac{1}{2}.$$
\end{lema}
\begin{proof}[Proof of Theorem \ref{permu}]
Conditioning on $S_{Cd}$ being not bad gives
 $$\E\sup_{s\in\ccF}\sum_{j=1}^n\mu_j^s S_{Cd}(j)\geq\frac{1}{2}\P(S_{Cd}\;\;\mbox{is  not bad})\geq\frac{1}{4},$$
 where in the last inequality we use Lemma \ref{key2}. 
\end{proof}

\noindent The rest of this section is devoted to the proof of Lemma \ref{key2}. Observe that without loss of generality we can normalize $\mu^{s}_j$ for every $s\in\ccF$ so that $\sum_{j=1}^n\mu^s_j s(j)=1$. In particular, $0\leq\mu^s_j\leq 1$.
\begin{lema}\label{lematepsilon}
Fix $r\in\ccF$ and $s\in\ccM_{Cd}$ which is bad. Then, there exists $\varepsilon(r,s) \in [0,1]$ such that
\begin{equation}\label{epsilony'}
\sum^{n}_{j=1} \1_{\{\mu^r_j> \va(r,s)\}}s(j)  \leq c\sum^{n}_{j=1} \1_{\{\mu^r_j> \va(r,s)\}}r(j).
\end{equation}
\end{lema}
\begin{proof}
Define
\begin{equation}\label{F}
F_{r,s}(\va):=\sum^{n}_{j=1}\rbr{\mu^r_j\wedge \va }s(j)
- c \sum^{n}_{j=1} \rbr{\mu^r_j\wedge \va }r(j).
\end{equation}
Obviously,  $F_{r,s}(0)=0$ and  since $s$ is bad we have 
\begin{align*}
F_{r,s}(1)=\sum_{j=1}^n\mu_j^r s(j)-c\sum_{j=1}^n\mu^r_j r(j)<c-c=0.
\end{align*}
Define $\va(r,s):=\sup\{\va\in[0,1):F_{r,s}(\va)\geq0\}$.
Following the proof of Lemma \ref{lema1}, we check that $\va(r,s)$ satisfies \eqref{epsilony'}.
\end{proof}
\begin{defi}
    For fixed $r\in\ccF$ and $s\in\ccM_{Cd}$ bad, we define $\va(r,s)$ as the largest number such that $F_{r,s}(\va(r,s))\geq 0$ (recall \eqref{F}) i.e. the largest number for which it holds that
\begin{equation}\label{epsilony}
\sum^{n}_{j=1} \1_{\{\mu^r_j> \va(r,s)\}}s(j)  \leq c\sum^{n}_{j=1} \1_{\{\mu^r_j> \va(r,s)\}}r(j).
\end{equation}
\end{defi}
\begin{defi}[Witness]\label{witness}
Fix $r\in\ccF$ and $s\in\ccM_{Cd}$ bad. Let $A(r,s):=\{j\in [n]:\; \mu^r_j>\va(r,s)\}$ and 
\[
 j(r,s):=\sum_{j\in A(r,s)}r(j).
\]
We say that $\hat{r}\in\ccF$ is admissible for $(r,s)$ if 
\begin{equation}\label{admis}
\forall_{j\in A(\hat{r},s)}\; \hat{r}(j)-s(j)\leq r(j).
\end{equation}
Let 
\[
t(\hat{r},s):=\sum_{j\in A(\hat{r},s)}(\hat{r}(j)-s(j))_{+}=\sum_{j\in A(\hat{r},s)}\max\left(\hat{r}(j)-s(j),0\right).
\]
We choose from all $\hat{r}\in \ccF$ admissible for $(r,s)$ the one with the smallest $j(\hat{r},s)$ and then with the smallest $t(\hat{r},s)$. We denote the chosen element (any) as $r_{\ast}$ and call it a witness for $(r,s)$ (with a small abuse of notation).
\end{defi}
\begin{rema}\label{j>t}
In the same way as for \eqref{ogrt}, we deduce from Lemma \ref{lematepsilon} that
\[
j(r_{\ast}, s)\geq t(r_{\ast},s) \geq (1-c)j(r_{\ast}, s).
\]
\end{rema}
\noindent Before we present the cover of $\ccF$ we need some notation. For two multisets $x,y \in [n]^d$ we write $x+y$ for the multiset $(x(j)+y(j))_j$, $x-y$ for $(\max\{x(j)-y(j),0\})_j$ and for some $A \subset [n]$ we write $x\1_{A}$ to denote the multiset with entries $x(j)$ if $j\in A$ and $0$ otherwise.
We also write $x\leq y$ if $x(j)\leq y(j)$ for every $j$.\\ 
\noindent
We construct the cover of $\ccF$ with multisets $(r_{\ast}-s)\1_{A(r_{\ast},s)}$, where we first fix bad $s\in\ccM_{Cd}$ and then we proceed with the construction of $r_{\ast}$ as in the Definition \ref{witness}. These are analogues of the sets $W(I,X)\setminus X$ of Section $3$. Define
\begin{equation}\label{cover1}
    \ccG(s):=\{(r_{\ast}-s)\1_{A(r_{\ast},s)}:r\in\ccF\}.
    \end{equation}
We see that $\ccG(s)$ is a cover of $\ccF$, because if $j\in A(s,r_{\ast})$, then $(r_{\ast}(j)-s(j))\1_{A(r_{\ast},s)}\leq r(j)$ by \eqref{admis} and otherwise $(r_{\ast}(j)-s(j))\1_{A(r_{\ast},s)}=0\leq r(j)$.
\begin{lema}\label{lemat 1}
Let $s, s'\in\ccM_{Cd}$ be bad and let $r,r'\in\ccF$. Suppose that $r_{\ast}$ is a witness of $(r,s)$ and that $r_{\ast}'$ is a witness of $(r',s')$. Assume that
\begin{enumerate}
\item $\tilde{j}:=j(r,s)=j(r',s')$,
\item $z:=s+(r_{\ast}-s)\1_{A(r_{\ast},s)}=s'+(r'_{\ast}-s')\1_{A(r'_{\ast}, s')}$,
\item $t:=t(r,s)=t(r',s')$.
\end{enumerate} 
Then, $\va(r_{\ast},s')\geq\va(r_{\ast}, s)$.
\end{lema}
\begin{proof}
We have to prove that $F_{r_{\ast},s'}(\va(r_{\ast},s))\geq 0$ (recall the definition \eqref{F}). Since $\va(r_{\ast}, s')$ is the largest number for which $F_{r_{\ast}, s'}$ is non-negative, we then conclude that $\va(r_{\ast}, s')\geq \va(r_{\ast}, s)$. By the definition of $A(r_{\ast}, s)$ and by adding and subtracting $s(j)$ we obtain
\begin{align*}
&\sum_{j=1}^n(\mu_j^{r_{\ast}}\wedge\va(r_{\ast},s))s'(j)=\va(r_{\ast},s)\sum_{j\in A(r_{\ast}, s)}s'(j)+\sum_{j\in A^c(r_{\ast}, s)}\mu_j^{r_{\ast}}s'(j)\\
&=\va(r_{\ast},s)\sum_{j\in A(r_{\ast}, s)}(s'(j)-s(j))+\va(r_{\ast},s)\sum_{j\in A(r_{\ast}, s)}s(j)\\
&+\sum_{j\in A^c(r_{\ast}, s)}\mu_j^{r_{\ast}}(s'(j)-s(j))+\sum_{j\in A^c(r_{\ast}, s)}\mu_j^{r_{\ast}}s(j)\\
&=\sum_{j=1}^n(\mu_j^{r_{\ast}}\wedge\va(r_{\ast},s))s(j)+\va(r_{\ast},s)\sum_{j\in A(r_{\ast}, s)}(s'(j)-s(j))+\sum_{j\in A^c(r_{\ast}, s)}\mu_j^{r_{\ast}}(s'(j)-s(j))\\
&=I+II+III.
\end{align*}
Notice that $I\geq c\sum_{j=1}^n(\va(r_{\ast},s)\wedge\mu_{j}^{r_{\ast}})r_{\ast}(j)$, so it is enough to show that
$II+III\geq 0$. We have that
\begin{align*}
&II=\va(r_{\ast},s)\sum_{j\in A(r_{\ast}, s)\cap A(r'_{\ast}, s')}(s'(j)-s(j))+\va(r_{\ast},s)\sum_{j\in A(r_{\ast}, s)\cap A^{c}(r'_{\ast}, s')}(s'(j)-s(j))\\
&=A+B.
\end{align*}
Now, assumption 2 can be written as
$$z=s\1_{A^{c}(r_{\ast},s)}+(r_{\ast}\vee s)\1_{A(r_{\ast},s)}=s'\1_{A^{c}(r'_{\ast},s')}+(r'_{\ast}\vee s')\1_{A(r'_{\ast},s')}.$$
Observe that for $j\in A(r_{\ast}, s)\cap A(r'_{\ast}, s')$, by the assumption $2$, we have that $r_{\ast}'(j)\vee s'(j)=r_{\ast}(j)\vee s(j)$, so we can write
$$A=\va(r_{\ast},s)\sum_{j\in A(r_{\ast}, s)\cap A(r'_{\ast}, s')}(s'(j)-(r_{\ast}'(j)\vee s'(j))+(r_{\ast}(j)\vee s(j))-s(j)).$$
For $j\in A(r_{\ast}, s)\cap A^{c}(r'_{\ast}, s')$, by the assumption $2$, we have that $s'(j)=r_{\ast}(j)\vee s(j)$, so
$$B=\va(r_{\ast},s)\sum_{j\in A(r_{\ast}, s)\cap A^{c}(r'_{\ast}, s')}((r_{\ast}(j)\vee s(j))-s(j)).$$
Therefore,
\begin{align*}
&II=A+B\\
&=\va(r_{\ast},s)\sum_{j\in A(r_{\ast}, s)}((r_{\ast}(j)\vee s(j))-s(j))+\va(r_{\ast},s)\sum_{j\in A(r_{\ast}, s)\cap A(r'_{\ast}, s')}(s'(j)-(r_{\ast}'(j)\vee s'(j)))\\
&=\va(r_{\ast},s)\sum_{j\in A(r_{\ast}, s)}(r_{\ast}(j)-s(j))_{+}-\va(r_{\ast},s)\sum_{j\in A(r_{\ast}, s)\cap A(r'_{\ast}, s')}(r'_{\ast}(j)-s'(j))_{+}.
\end{align*}
Similarly, by noticing that $s(j)=s'(j)$ whenever $j\in A^{c}(r_{\ast},s)\cap A^{c}(r'_{\ast},s')$, by the assumption 2, we get that
\begin{align*}
III&=\sum_{j\in A^c(r_{\ast}, s)}\mu_j^{r_{\ast}}(s'(j)-s(j))=\sum_{j\in A^c(r_{\ast}, s)\cap A^c(r'_{\ast}, s')}\mu_j^{r_{\ast}}(s'(j)-s(j))\\
&+\sum_{j\in A^c(r_{\ast}, s)\cap A(r'_{\ast}, s')}\mu_j^{r_{\ast}}(s'(j)-s(j))\\
&=\sum_{j\in A^c(r_{\ast}, s)\cap A(r'_{\ast}, s')}\mu_j^{r_{\ast}}(s'(j)-s(j)).
\end{align*}
On the set $A^c(r_{\ast}, s)\cap A(r'_{\ast}, s')$, by the assumption 2, we have $s(j)=(r'_{\ast}(j)\vee s'(j))$, so
\begin{align*}
III&=\sum_{j\in A^c(r_{\ast}, s)\cap A(r'_{\ast}, s')}\mu_j^{r_{\ast}}(s'(j)-s(j))=\sum_{j\in A^c(r_{\ast}, s)\cap A(r'_{\ast}, s')}\mu_j^{r_{\ast}}(s'(j)-(r'_{\ast}(j)\vee s'(j)))\\
&=-\sum_{j\in A^c(r_{\ast}, s)\cap A(r'_{\ast}, s')}\mu_j^{r_{\ast}}(r'_{\ast}(j)-s'(j))_{+}.
\end{align*}
Finally,
\begin{align*}
II+III=\va(r_{\ast},s)\sum_{j\in A(r_{\ast}, s)}(r_{\ast}(j)-s(j))_{+}&-\va(r_{\ast},s)\sum_{j\in A(r_{\ast}, s)\cap A(r'_{\ast}, s')}(r'_{\ast}(j)-s'(j))_{+}\\
&-\sum_{j\in A^c(r_{\ast}, s)\cap A(r'_{\ast}, s')}\mu_j^{r_{\ast}}(r'_{\ast}(j)-s'(j))_{+}.
\end{align*}
For $j\in A^{c}(r_{\ast},s)$, $\mu_j^{r_{\ast}}\leq\va(r_{\ast}, s)$. Hence, by the assumption 3,
\begin{align*}
II+III&\geq\va(r_{\ast},s)\sum_{j\in A(r_{\ast}, s)}(r_{\ast}(j)-s(j))_{+}-\va(r_{\ast},s)\sum_{j\in A(r'_{\ast}, s')}(r'_{\ast}(j)-s'(j))_{+}\\
&=\va(r_{\ast},s)t-\va(r_{\ast},s)t=0.
\end{align*}
\end{proof}
\begin{lema}\label{lemat 2}
Under the assumptions of Lemma \ref{lemat 1} it holds that
\begin{equation}\label{zawieranie1}
    (r_{\ast}-s')\1_{A(r_{\ast},s')}=(r'_{\ast}-s')\1_{A(r'_{\ast},s')}.
\end{equation}
\end{lema}
\begin{proof}
In order to prove \eqref{zawieranie1} we need to prove that $\va(r_{\ast},s')=\va(r_{\ast},s)$. Due to Lemma \ref{lemat 1} it is enough to show that $\va(r_{\ast},s')\leq\va(r_{\ast},s)$. Assume for contradiction that $\va(r_{\ast},s')>\va(r_{\ast},s)$. 
By the same argument as in the proof of Lemma \ref{zaw}, we deduce that  $A(r_{\ast},s')\subset A(r_{\ast},s)$ and, in turn, that $j(r_{\ast},s')<j(r_{\ast},s)$. We have two cases. Either $j\in A(r_{\ast},s')\cap A^{c}(r'_{\ast},s')$ or $j\in A(r_{\ast},s')\cap A(r'_{\ast},s')$.\\ 
\noindent \textbf{Case 1.} Suppose that $j\in A(r_{\ast},s')\cap A^{c}(r'_{\ast},s')$. Since, $A(r_{\ast},s')\subset A(r_{\ast},s)$ and by the assumption $2$, we have
$$r_{\ast}(j)\vee s(j)=s'(j).$$
In particular, $s'(j)\geq r_{\ast}(j)$, so 
$$(r_{\ast}(j)-s'(j))_{+}=0\leq r'_{\ast}(j).$$
\textbf{Case 2.} Suppose that $j\in A(r_{\ast},s')\cap A(r'_{\ast},s')$. Again, since $A(r_{\ast},s')\subset A(r_{\ast},s)$ and by the assumption $2$, we have 
$$r_{\ast}(j)\vee s(j)=r'_{\ast}(j)\vee s'(j).$$
We have two cases. If $s'(j)\geq r_{\ast}'(j)$, then $(r_{\ast}(j)\vee s(j))=s'(j)$, so $r_{\ast}(j)\leq r_{\ast}(j)\vee s(j)=s'(j)$, and hence, as before, 
$$(r_{\ast}(j)-s'(j))_{+}=0\leq r'_{\ast}(j).$$
If $s'(j)<r'_{\ast}(j)$, then $r'_{\ast}(j)\geq r_{\ast}(j)$, which implies that 
$$(r_{\ast}(j)-s'(j))_{+}\leq(r'_{\ast}(j)-s'(j))_{+}\leq r_{\ast}'(j).$$
In conclusion we have that for $j\in A(r_{\ast},s')$ it holds that 
$$(r_{\ast}(j)-s'(j))_{+}\leq r'_{\ast}(j),$$
which means that $r_{\ast}$ is admissible for $(r',s')$. Moreover, we have $j(r_{\ast},s')<j(r_{\ast},s)$. But, by the assumption $1$, $j(r_{\ast},s)=j(r'_{\ast},s')$, which contradicts the assumption that $r'_{\ast}$ is a witness for $(r',s')$ i.e. that $j(r'_{\ast},s')$ is the smallest possible, in particular, that $j(r'_{\ast},s')\leq j(r_{\ast},s')$. Hence, $\va(r_{\ast},s')\leq\va(r_{\ast},s)$, which together with Lemma \ref{lemat 1} gives $\va(r_{\ast},s')=\va(r_{\ast},s)$.\\
\noindent 
The main consequence of the fact proved above is that $A(r_{\ast},s')=A(r_{\ast},s)$. \\
\noindent To prove \eqref{zawieranie1} we again consider two cases.\\
\noindent \textbf{Case 1.} Suppose that $j\in A(r_{\ast},s)\cap A^{c}(r'_{\ast},s')$. Then, by the assumption 2, we have that $s'(j)=s(j)\vee r_{\ast}(j)$ and, in turn,
$$(r_{\ast}(j)-s'(j))_{+}\leq (r_{\ast}(j)\vee s(j)-s'(j))_{+}=0.$$
\textbf{Case 2.} Suppose that $j\in A(r_{\ast},s)\cap A(r'_{\ast},s')$. So, by the assumption 2, $r_{\ast}(j)\vee s(j)=r'_{\ast}(j)\vee s'(j)$ and
\begin{align*}
   (r_{\ast}(j)-s'(j))_{+}&\leq (r_{\ast}(j)\vee s(j)-s'(j))_{+}\\
   &=(r'_{\ast}(j)\vee s'(j)-s'(j))_{+}\\
   &=(r'_{\ast}(j)-s'(j))_{+}.
\end{align*}
So, since $A(r_{\ast},s')=A(r_{\ast},s)$, we can deduce that 
$$(r_{\ast}-s')\1_{A(r_{\ast},s')}\leq(r'_{\ast}-s')\1_{A(r_{\ast},s')\cap A(r'_{\ast},s')}\leq (r'_{\ast}-s')\1_{ A(r'_{\ast},s')}.$$
But, $r_{\ast}'$ provides minimal $t(r_{\ast}',s')$, we therefore conclude that
$$(r_{\ast}-s')\1_{A(r_{\ast},s)}=(r_{\ast}-s')\1_{A(r_{\ast},s')}=(r'_{\ast}-s')\1_{A(r'_{\ast},s')}.$$
\end{proof}
\begin{coro}\label{wniosek1}
Fix a positive integer $t$, $z\in\ccM_{Cd+t}$ and bad $s\in\ccM_{Cd}$ such that $s\leq z$. Then, for any $\tilde{j}\geq t$
\begin{equation}\label{jpot}
    |\{(r_{\ast}-s)\1_{A(r_{\ast},s)}: r\in\ccF,\; j(r_{\ast},s)=\tilde{j},\; z=s+(r_{\ast}-s)\1_{A(r_{\ast},s)},\; t(r_{\ast},s)=t\}|\leq\binom{\tilde{j}}{t}.
\end{equation}
\end{coro}
\begin{proof}
    Consider any pair $(r,s)$ satisfying \eqref{jpot} where $s\in\ccM_{Cd}$ is bad and any other pair $(r',s')$ satisfying \eqref{jpot} where $s'\in\ccM_{Cd}$ is bad. By Lemma \ref{lemat 2} we have 
    $$(r_{\ast}'-s')\1_{A(r_{\ast}',s')}=(r_{\ast}-s')\1_{A(r_{\ast},s)}\leq r_{\ast}\1_{A(r_{\ast},s)},$$
so there are at most $\binom{\tilde{j}}{t}$ such multisets.

\end{proof}
\begin{proof}[Proof of Lemma \ref{key2}]
We are ready for the final computation. The strategy is the same as in the proof of Lemma \ref{zly}. We want to replace bad $s\in\ccM_{Cd}$ by suitably chosen $z\in\ccM_{Cd+t}$. Fix $\tilde{j}, t, z$ from Lemma \ref{lemat 1}. Consider $w(r',s'):=(r_{\ast}'-s')\1_{A(r_{\ast}',s')}$ and 
$w(r,s)=(r_{\ast}-s)\1_{A(r_{\ast},s)}$ for which the parameters $\tilde{j},t,z$ are the same. Fix a single $r(\tilde{j},t,z):=r_{\ast}1_{A(r_{\ast},s)}\leq z$. By Corollary \ref{wniosek1} we know that $w(r',s')\leq r(\tilde{j},t,z)$. 
Note that $s=z-w(r_{\ast},s)$.
We use the notation $|w|:=\sum_{j=1}^n w(j)$ for any multiset $w$. Recall the cover $\ccG(s)$ (see \eqref{cover1}) of $\ccF$. Since we assume that $\ccF$ is not small, we have for any bad $s\in\ccM_{Cd}$ that $\sum_{w(r,s)\in\ccG(s)}A(w(r,s))\geq \frac{1}{2}$ (recall Definition \ref{A(w)}). We can therefore write
$$\frac{1}{2}\P(S_{Cd}\;\;\mbox{is bad})=\frac{1}{2}\sum_{\substack{s\in\ccM_{Cd} \\s\; \textrm{is bad}}}\P(S_{Cd}=s)\leq \sum_{\substack{s\in\ccM_{Cd} \\s\; \textrm{is bad}}}\P(S_{Cd}=s)\sum_{w(r,s)\in \ccG(s)}A(w(r,s)).$$
The crucial step, which is a result of the above discussion and Remark \ref{j>t} is that we can write
\begin{multline*}
 \sum_{\substack{s\in\ccM_{Cd} \\s\; \textrm{is bad}}}\P(S_{Cd}=s)\sum_{w(r,s)\in \ccG(s)}A(w(r,s))\\
 \leq\sum_{\tilde{j}\geq1} \sum_{(1-c)\tilde{j}\leq t\leq \tilde{j}} \sum_{z\in\ccM_{Cd+t}}\sum_{\substack{w\leq  r(\tilde{j},t,z)\\ |w|=t}} \P(S_{Cd}=z-w)A(w).   
\end{multline*}
Now, by \eqref{defmulti}, Definition \ref{A(w)} and since $w\leq z$ we get 

\begin{align*}
\P(S_{Cd}=z-w)A(w)&=\frac{(Cd)!\prod^{n}_{j=1}p(j)^{z(j)-w(j)}}{\prod^n_{j=1} (z(j)-w(j))!} \frac{(ed)^t\prod^n_{j=1}p(j)^{w(j)}}{\prod^n_{j=1}w(j)!}\\
&=(Cd)!(ed)^t\prod_{j=1}^n\frac{p(j)^{z(j)}}{(z(j)-w(j))!w(j)!},
\end{align*}
where we used that $\sum_{j=1}^n w(j)=t$. Now, since $\P(S_{Cd+t}=z)=(Cd+t)!\prod_{j=1}^n\frac{p(j)^{z(j)}}{z(j)!}$, we have
$$\prod_{j=1}^n\frac{p(j)^{z(j)}}{(z(j)-w(j))!w(j)!}=\frac{1}{(Cd+t)!}\P(S_{Cd+t}=z)\prod_{j=1}^n\binom{z(j)}{w(j)}.$$
Because of \eqref{epsilony}, we can bound 
$$\sum_{j\in A(r_{\ast},s)}z(j)=\sum_{j\in A(r_{\ast},s)}s(j)\vee r_{\ast}(j)=\sum_{j\in A(r_{\ast},s)}(s(j)+(r_{\ast}(j)-s(j))_{+})\leq c\tilde{j}+t\leq \tilde{j}+t$$
since $c=\frac{1}{2}$. So, using the fact that $w(j)=0$ for $j\notin A(r_{\ast},s)$,
\[
\prod^n_{j=1}\binom{z(j)}{w(j)} =\prod_{j\in A(r_{\ast},s)}\binom{z(j)}{w(j)} \leq \prod_{j\in A(r_{\ast},s)}2^{z(j)}\leq2^{\tilde{j}+t}.
\]
So, we have
\begin{align*}
    &\sum_{\tilde{j}\geq1} \sum_{(1-c)\tilde{j}\leq t\leq \tilde{j}} \sum_{z\in\ccM_{Cd+t}}\sum_{\substack{w\leq  r(\tilde{j},t,z)\\ |w|=t}} \P(S_{Cd}=z-w)A(w)\\
    &\leq \sum_{\tilde{j}\geq1} \sum_{(1-c)\tilde{j}\leq t\leq \tilde{j}} \sum_{z\in\ccM_{Cd+t}}\sum_{\substack{w
    \leq  r(\tilde{j},t,z)\\ |w|=t}} 2^{\tilde{j}+t}(ed)^t \frac{(Cd)!}{(Cd+t)!}\P(S_{Cd+t}=z).
\end{align*}
Now, by Corollary \ref{wniosek1}, we have that
\[   
\sum_{\substack{w\leq  r(\tilde{j},t,z)\\
|w|=t}}1\leq\binom{\tilde{j}}{t},
\]
which together with the fact that $\frac{(Cd)!}{(Cd+t)!}\leq (Cd)^{-t}$ gives
\begin{align*}
    &\sum_{\tilde{j}\geq1} \sum_{(1-c)\tilde{j}\leq t\leq \tilde{j}} \sum_{z\in\ccM_{Cd+t}}\sum_{\substack{w
    \leq  r(\tilde{j},t,z)\\ |w|=t}} 2^{\tilde{j}+t}(ed)^t \frac{(Cd)!}{(Cd+t)!}\P(S_{Cd+t}=z)\\
    &\leq \sum_{\tilde{j}\geq1} \sum_{(1-c)\tilde{j}\leq t\leq \tilde{j}} \sum_{z\in\ccM_{Cd+t}} \binom{\tilde{j}}{t} 2^{\tilde{j}+t}(ed)^t (Cd)^{-t} \P(S_{Cd+t}=z)\\
    &=\sum_{\tilde{j}\geq1} \sum_{(1-c)\tilde{j}\leq t\leq \tilde{j}} \binom{\tilde{j}}{t} 2^{\tilde{j}+t}\rbr{e/C}^t,
\end{align*}
where the last equality follows since we are summing up the probability mass function of $S_{Cd+t}$. Now, by substituting $\bar{t}:=\tilde{j}-t$, we get
 \begin{align*}
 &\sum_{\tilde{j}\geq1} \sum_{(1-c)\tilde{j}\leq t\leq \tilde{j}} \binom{\tilde{j}}{t} 2^{\tilde{j}+t}\rbr{e/C}^t= \sum_{\tilde{j}\geq1} 2^{\tilde{j}}  2^{\tilde{j}}  (e/C)^{\tilde{j}} \sum_{0\leq t\leq c\tilde{j}} \binom{\tilde{j}}{\bar{t}} 2^{-\bar{t}}(e/C)^{-\bar{t}}.
 \end{align*}
We bound $\binom{\tilde{j}}{\bar{t}}\leq 2^{\tilde{j}}$, $2^{-\bar{t}}\leq 1$ and
$$\sum_{0\leq\bar{t}\leq c\tilde{j}}(e/C)^{-\bar{t}}\leq\frac{(C/e)^{c\tilde{j}+1}}{C/e-1}.$$
Hence,
\begin{align*}
&\sum_{\tilde{j}\geq1} 2^{\tilde{j}}  2^{\tilde{j}}  (e/C)^{\tilde{j}} \sum_{0\leq\bar{t}\leq c\tilde{j}} \binom{\tilde{j}}{\bar{t}} 2^{-\bar{t}}(e/C)^{-\bar{t}}\\
&\leq \frac{C/e}{C/e-1}\sum_{\tilde{j}\geq1} 8^{\tilde{j}} (e/C)^{(1-c)\tilde{j}}.
 \end{align*}  
Finally, for $C=[41^2e]$  and since $c=1/2$ we get
\begin{align*}
& \frac{1}{2}\P(S_{Cd}\;\;\mbox{is bad})\leq\frac{41^2}{41^2-2}\cdot\frac{8}{\sqrt{41^2-1}-8}\leq\frac{1}{4}.
\end{align*}
Hence,
\[
\P(S_{Cd}\;\;\mbox{is bad})\leq \frac{1}{2}.
\]
\end{proof}
\section*{Acknowledgements}
The authors would like to thank the organizers of the "Workshop in Convexity and High-Dimensional Probability" held on May 23-- 27, 2022 at Georgia Institute of Technology at which some progress on this paper was made.\\
We are very grateful to anonymous referees for the careful reading of the initial version of this paper and pointing out all the needed clarifications.

\end{document}